\documentclass[12pt, leqno]{article}

\usepackage{amsmath,amssymb,amsthm}
\usepackage{latexsym}
\usepackage[small,nohug,heads=littlevee]{diagrams}
\diagramstyle[labelstyle=\scriptstyle]
\usepackage[all]{xy}

\usepackage[dvips]{graphicx}
\newarrow{Into} C--->

 \newtheorem{thm}{Theorem}
 
 \newtheorem{lem}{Lemma}
 
 \newtheorem{prop}{Proposition}
 \newtheorem{defi}{Definition}

 \theoremstyle{remark}
 \newtheorem{rem}{Remark}[section]

\newcommand{\e}{\leqno}

\usepackage[dvips]{graphicx}
\usepackage{graphics}
\usepackage{amsmath}
\usepackage{amscd}
\usepackage{amssymb}

\title{Topics in Geometric Group Theory. Part I}

\author{Daniele Ettore Otera\footnote{Vilnius University, Institute of Data Science and Digital Technologies,
Akademijos st. 4, LT-08663, Vilnius, Lithuania.
e-mail: daniele.otera@mii.vu.lt \ - \  daniele.otera@gmail.com}
 \textsc{ and } Valentin Po\'enaru\footnote{Professor Emeritus at
Universit\'e Paris-Sud 11, D\'epartement de Math\'ematiques,
B\^atiment 425, 91405 Orsay, France. e-mail: valpoe@hotmail.com }}

\begin{document}

\maketitle

\begin{abstract}
This survey paper concerns mainly with some asymptotic topological properties of finitely presented discrete groups:  \textit{quasi-simple filtration} (\textsc{qsf}), \textit{geometric simple connectivity} (\textsc{gsc}), \textit{topological inverse-\textsc{representations}}, and  the
notion of \textit{easy groups}. As we will explain, these properties are central in the theory of discrete groups seen from a topological viewpoint at infinity.
Also, we shall outline the main steps of the achievements obtained in the last years around the very general question whether or not all finitely presented groups satisfy these conditions.

\vspace{0.1cm} \noindent {\bf Keywords:} Geometric simple
connectivity (\textsc{gsc}), quasi-simple filtration ({\sc qsf}),
inverse-\textsc{representations}, finitely presented groups, universal
covers, singularities, Whitehead manifold, handlebody decomposition, almost-convex groups, combings, easy groups.

\vspace{0.1cm}  \noindent {\bf MSC Subject:} 57M05; 57M10; 57N35; 57R65; 57S25; 20F65; 20F69.
\end{abstract}

\tableofcontents

\section{Introduction}

             This is a two-part survey paper, part I by D. Otera and V. Po\'enaru, part II by V. Po\'enaru (to appear soon in arXiv).  In this first part we present, mostly, background material, historical context, and motivations. Its main items are the theory of \textit {topological inverse}-\textsc{representations}, a sort of dual notion to the usual group representations, some tameness conditions at infinity for open manifolds, cell-complexes and discrete groups, such as \textit{geometric simple connectivity}  (\textsc{gsc}), \textit{Dehn-exhaustibility}, the \textsc{qsf} {\textit {property}}, and then a glimpse into the more recent joint work of both authors on combings of groups  and on the so-called {\textit {Whitehead nightmare}}.

          On the other hand, Part II is much more technical, since it outlines the proof  (announced in  \cite{Po_QSF_survey} and proved in \cite{Po_QSF2, Po_QSF3})  that all finitely presented  groups satisfy the \textsc{qsf} property (see also the related \cite{Po_QSF1_Geom-Ded}).  This highly non-trivial notion, which in a certain sense is the good extension of simple connectivity at infinity (they are actually equivalent for closed 3-manifold groups), was introduced by Brick, Mihalik  and Stallings \cite{BM1, St3}, and it has roots in earlier work of Casson and of Po\'enaru \cite{Ger-St, Po3-JDG}, as well as in the much older work of Max Dehn.

          Note that, even for the very special class of 3-manifold groups, in order to show they all are {\textsc{qsf}}  one needs deep and difficult results such as either the geometrization theorem in dimension 3 \cite{BBBMP, Mo-Ti, Mo-Ti2, Per1, Per2, Per3, Thu} or else the works of Agol, Haglund and Wise on special groups \cite{Ag, HW}.

          \subsection{Acknowledgments} We would like to thank the referee who helped us to greatly improve the presentation of the paper. His/her pertinent suggestions, corrections and comments gave us the opportunity to make our paper much more readable, precise, structured  and complete.

\section{Topology at infinity, universal covers and discrete groups}

The main technical tool used in this survey is what we call a topological inverse-\textsc{representation} for a finitely presented discrete group $G$. We write here ``\textsc{representation}'' in capital letters, so as to distinguish it from the standard well-known group representations. But some preliminaries will be necessary before we can actually define our \textsc{representations} of groups.

\subsection{The $\Phi/\Psi$-theory and zippings}
We consider some completely general \textbf{non-degenerate} simplicial map $f: X \to M$.
Here $X$ and $M$ are simplicial complexes, $X$ being at most countable, and it is also assumed that ${\rm{dim }} (X)\leq {\rm{dim }} (M)$. But $X$ need not be locally-finite and it will be endowed with the \textbf{weak topology}. By definition, a \textbf{\textit{singularity}} $x \in {\rm {Sing}} (f) \subset X$ is a point of $X$ in the neighborhood of which $f$ fails to inject, i.e. fails to be immersive.

Trivially, the map $f$ defines an equivalence relation $\Phi(f) \subset X \times X$, by $(x_1, x_2) \in \Phi (f) \Leftrightarrow f(x_1)= f(x_2)$.

We will be interested in equivalence relations on $X$, $\mathcal R \subset X\times X$, with the following additional feature:
 let $\sigma_1, \sigma_2 \subset X$ be two simplexes of the same dimension and such that $f(\sigma_1)=f(\sigma_2)$. If we can find $x_1 \in \rm{int} \ \sigma_1$, $x_2 \in \rm{int} \ \sigma_2$
 such that $f(x_1)= f(x_2)$, then we also have the implication $(x_1, x_2) \in \mathcal R \Rightarrow \mathcal R$ identifies $\sigma_1$ to $\sigma_2$.

 This kind of equivalence relations will be called \textbf{\textit{$f$-admissible}}. Notice that, whenever $\mathcal R \subset \Phi (f)$ is admissible, we have a natural simplicial commutative diagram
$$\begin{diagram}
X &&\rTo^f(4,2)&&M \\
    &\rdTo_{\pi(\mathcal R)}& &\ruTo_{f_1(\mathcal R)}\\& &X/\mathcal R.
\end{diagram} \e (1-1) $$

With this, we develop now the following little theory, for which the detailed definitions, lemmas and proofs, are to be found in \cite{Po1-duke, Po_QSF1_Geom-Ded}. The point here is that, in the context of our non-degenerate simplicial map $f$, there is also another equivalence relation more subtle than the $\Phi(f)$. This is the $\Psi(f) \subset \Phi(f) \subset X\times X$, characterized by the following lemma:

\begin{lem}(\cite{Po1-duke}) \label{lemma1}
There is a {\textbf{unique}}, $f$-admissible equivalence relation $\Psi (f) \subset \Phi (f) \subset X \times X $ which has the following two properties, which also characterize it.
\begin{itemize}
\item[(i)]  In the context of the corresponding commutative diagram (1 -- 1), which we write again below
$$\begin{diagram}
X &&\rTo^f(4,2)&&M \\
    &\rdTo_{\pi= \pi(\Psi (f))}& &\ruTo_{f_1= f_1 (\Psi(f))}\\& &X/ \Psi(f)
\end{diagram} \e (1-2) $$
we have ${\rm{Sing}} (f_1) = \emptyset$, i.e. $f_1$ is \textbf{immersive}.
\item[(ii)] Let $\mathcal R$ be \textbf{any}  equivalence relation (not necessarily assumed a priori to be admissible), such that  $\mathcal R\subset \Phi(f)$ and  ${\rm{Sing}} (f_1(\mathcal R)) = \emptyset$.   Then  $\mathcal R= \Psi (f)$.
\end{itemize}
\end{lem}

In plain English, $\Psi(f)$ is the smallest equivalence relation, compatible with $f$, which kills all the singularities. We also have the next lemma, which has no equivalent for $\Phi(f)$.

\begin{lem}(\cite{Po1-duke}) \label{lemma2}
In the context of (1 -- 2), the following induced map is surjective: $$\pi _\ast : \pi_1 (X) \to \pi_1 \big (X/ \Psi (f) \big ).$$
\end{lem}

Of course, nothing like Lemma \ref{lemma2} is true for the equivalence relation $\Phi(f)$ which, contrary to $\Psi(f)$, has no topological memory, a feature which is one of the reasons
that  makes $\Psi(f)$ interesting.

We will give now some ideas of how one may construct $\Psi(f)$. In order to effectively build our $\Psi(f)$, let us look now for a most efficient way of killing the singularities of $f$, in a manner which should have the minimum of collateral effects.

Start with two simplexes $\sigma _1,  \sigma _2 \subset X$ of the same dimension, such that $f( \sigma _1) = f( \sigma _2)$, and with
 $\sigma _1\cap  \sigma _2 \ni x_1 \in {\rm{Sing}} (f)$. We can go then to a first quotient space $f_1: X_1= X/\rho_1  \to M$, that kills $x_1$ via the \textbf{folding map} which identifies $\sigma _1$ to  $\sigma _2$. This is an equivalence relation $\rho _1 \subset  \Phi(f)$, and one can continue this process, with a sequence of folding maps:
 $$\rho_1 \subset \rho_2 \subset \ldots \subset \rho_n \subset \rho_{n+1} \ldots \subset \Phi(f).$$
 Here $\rho_{\omega} = \bigcup_1^{\infty} \rho_i$ is an equivalence relation too, subcomplex of  $\Phi(f)$ (i.e. closed in the weak topology), and the following induced map $$ f_{\omega} :  X / \rho_{\omega} \to M$$ is again simplicial. If $ f_{\omega}$ is not immersive, we can continue to the next transfinite ordinal, with $\rho_{\omega} \subset \rho_{\omega +1} \subset X \times X$, and go to:
 $$f_{\omega +1} : X / \rho_{\omega +1} \to M.$$

 Now, we can continue via the obvious transfinite induction. Since $X$ is assumed to be at most countable, the process has to stop at some countable ordinal $\omega_1$, and one can actually show that $\rho_{\omega_1}= \Psi(f)$. This $\omega_1$ is \textbf{not} unique, but then we also have the following useful lemma.
 \begin{lem}(\cite{Po1-duke})\label{lemma3}
One can \textbf{chose} our sequence of folding maps $(\rho _i)$ so that $\omega_1 = \omega$.
\end{lem}

The sequence of folding maps coming with $\rho_{\omega}= \Psi(f)$ is called a {\textbf{{\textit{zipping}}}} of $f$. This kind of \textbf{zipping strategy} is, of course, not unique. Towards the end of this paper, another alternative way to conceive zipping strategies will be presented too.

\subsection{Geometric simple connectivity}

The notion of {\textbf{{\textit{geometric simple connectivity}}} (\textsc{gsc}), which will play a big role in this survey, stems from differential topology, and was probably conceived by T. Wall \cite{Wall}. A smooth manifold $M$ is said to be \textsc{gsc} ({\textit{geometrically simply connected}})  if it has a handlebody decomposition without handles of index one (or, more precisely, if it possesses a smooth handlebody decomposition with a unique handle of index zero and with its 1-handles and 2-handles in cancelling position). If $M$ is closed, this can also be easily rephrased in Morse-theoretical terms. Saying that $M$ is \textsc{gsc} means that there exists a Morse function $f : M \to \mathbb R_+$ without singularities of index $\lambda =1$. Of course \textsc{gsc} can also be defined Morse-theoretically for open manifolds and for the case of non-empty boundary \cite{Po2-duke}. But then appropriate restrictions have to be added to the definition and we will not go into that here.  Concerning \textsc{gsc}, more details are also provided in \cite{FG, FO2, Ot_surv, Ot-Po-Ta, Po_GSC, Po_4-dim, PT-ActaHung, PT4-AMS}.

As said above, the \textsc{gsc} concept also extends to open manifolds, manifolds with boundary and to cell-complexes too, and the ones of interest here will always be infinite.

\begin{defi} \label{gsc} A cell-complex $X$ is said to be \textsc{gsc} if it admits a cell decomposition of the following type
$$X= T \ \bigcup \ \sum _{i=1}^\infty \{ \mbox{1-cells } H_i^1\}\ \bigcup \ \sum _{j=1}^\infty \{ \mbox{2-cells } H_j^2 \}\ \bigcup \  \sum _{k; \lambda \geq 2}\{
\lambda\mbox{-cells } H_k^{\lambda}\}, \e(1-3)$$ such that:

\begin{itemize}
\item $T\subset X$ is a \textsc{properly}{\footnote{In the present paper \textsc{proper}, written with capital letters, will mean inverse image of compact is compact, while ``proper'' will mean interior goes to interior and boundary goes to boundary.}} embedded tree,
\item in (1 -- 3) the infinite sets of indices $\{i\}$ and $\{j\}$ are in canonical bijection and the {\textbf{\emph{geometric intersection
matrix}}} takes the form
$$H_j^2 \cdot H_i^1= \{\delta_{ji} + a_{ji} \in \mathbb Z_+, \mbox { where }
 a_{ji}> 0 \Leftrightarrow j>i\}. \e(1-4)$$
Here $H_j^2 \cdot H_i^1$ is the number of times $H_j^2$ goes through $H_i^1$, without any $+$ or $-$ signs being counted. Notice that $\{H_j^2\}$ is only an appropriately chosen {\textbf{subset}} of the 2-cells $\{H_j^2\} \subsetneq \{H^2\}$.
\end{itemize}
\end{defi}

Of course, the decomposition (1 -- 3) also makes sense when $X$ is replaced by a smooth manifold $M^n$. If $M^n$ is open, then the $T$ in (1 -- 3) becomes a ball $B^n \setminus \{$a closed {\textbf{\textit{tame}}} subset deleted from $\partial B^n\}$, and ``tame'' means that the subset in question is contained inside some smooth line embedded in $\partial B^n \simeq S^{n-1}$; while when $M^n$ is compact, $T$ becomes a $B^n$ (a handle of index $\lambda =0$). Similarly, the various $\lambda$-cells occurring in (1 -- 3), always with $\lambda > 1$, become smooth $n$-handles of index $\lambda$ (and see here \cite{Po_QSF1_Geom-Ded, Po_GSC, Po_4-dim, PT-ActaHung} for more details).

Also, the condition in the right hand side of (1 -- 4) will be called \textbf{\textit{easy id + nilpotent}}. This should be distinguished from the dual condition ``\textbf{\textit{difficult id + nilpotent}}'' where the inequality in (1 -- 4) gets reversed into $a_{ji}> 0 \Leftrightarrow j < i$. Of course, in the finite case the two conditions are equivalent, but {\bf not} in our infinite case of interest. For instance, the classical {\it Whitehead manifold} $Wh^3$ \cite{Wh} (see the next section for details), always the villain in our present story, admits cell-decompositions which are of the difficult id + nilpotent type. And $Wh^3$ is certainly not \textsc{gsc}, a fact which is connected to the $\pi_1^{\infty} Wh^3 \neq 0$, as we shall see.

Here is now a useful related definition which is a sort of weakening of the \textsc{gsc}, specific for the non-compact case, due to L. Funar \cite{FG, FO2}.

\begin{defi}\label{wgsc}
An infinite cell-complex $X$ is said to be {\bf{weakly geometrically
simply connected}} (\textsc{wgsc}) if  it admits an exhaustion by compact and simply-connected
sub-complexes $X =
\bigcup_{i=1}^\infty K_i$, where $K_1 \subset K_2 \subset \ldots
\subset K_j \subset \ldots \subset X$ and $\pi_1 (K_i) =0$ for all $i$'s.
\end{defi}

Clearly \textsc{gsc} $\Rightarrow$ \textsc{wgsc} and in the compact case \textsc{wgsc} becomes mere simple-connectivity. One may also notice that, while \textsc{gsc} has close ties with the all-important notion of ``collapsibility'', there is nothing like that for \textsc{wgsc}.

Finally, we  remind the reader of another fundamental tameness condition at infinity for open manifolds. A locally compact and simply connected space $X$ is said to be \textbf{\textit{simply connected at infinity}} (and one will write $\pi _1 ^{\infty} X =0$) if for every compact $k \subset X$ there is a larger compact subset $k \subset K \subset X$ such that the induced map $\pi_1 (X -K) \to \pi _1 (X - k)$ is zero.

The simple connectivity at infinity is a very strong asymptotic notion and in fact,  it has been used to detect Euclidean spaces among open contractible manifolds, by results of Siebenmann and Stallings in dimension at least 5 \cite{Sie, St1}, Freedman in dimension 4 \cite{Free}, and Edwards and Wall in dimension 3 \cite{Ed, Wall_0} (see \cite{Ot_surv} for a general  overview).
Note also that the notion $\pi _1 ^{\infty} =0$ also makes sense for finitely presented groups \cite{Br1, Ta}.  An old classical easy fact specific for dimension three is the following one: an open 3-manifold $V^3$ which is \textsc{wgsc} is also simply connected at infinity (see e.g. \cite{Ot-Ru-Ta2}). In higher dimension ($n\geq 5$), it turns out that open simply connected $n$-manifolds which are also simply connected at infinity are actually \textsc{gsc} \cite{PT-ActaHung}.

M. Davis \cite{Da} exhibited, in any dimension at least 4,  the first examples of fundamental groups of aspherical closed manifolds which are {\bf {not}} simply connected at infinity, while in dimension 3, all finitely presented (3-manifold) groups are simply connected at infinity, as a consequence of Thurston geometrization conjecture \cite{Thu} proved by Perelman's work \cite{BBBMP, Mo-Ti, Mo-Ti2, Per1, Per2, Per3} (this also follows by Agol's proof of  the virtually Haken Conjecture \cite{Ag} and by the work of Haglund-Wise on special groups \cite{HW}).

\subsection{The Whitehead manifold}

Before going on, let us say few words about the renowned Whitehead 3-manifold $Wh^3$  and its cell-decompositions. This manifold is the main example of an open 3-manifold which is contractible but not homeomorphic to the Euclidean space $\mathbb R^3$ (see for instance \cite{N-Wh,Wh}); in particular, the manifold $Wh^3$ is homotopically equivalent to a point but is not \textbf{\textit{tame}} (which means that it is not homeomorphic to a compact manifold with a closed subset of the boundary removed, see e.g. \cite{Ag, MT, Tu}).

This interesting manifold was discovered by J.H.C. Whitehead around 1934 (\cite{Wh}), actually as a counterexample to his attempted proof of Poincar\'e Conjecture. Roughly speaking, its construction goes as follows: take a solid torus $T_1 \subseteq \mathbb R^3$, and embed a thinner torus $T_2$ inside $T_1$ so that it ''links with itself'' (or, in other simple words, so that it embeds as a {\it Whitehead link}, see e.g. \cite{FG, Gab_2}). Now, iterate the process in such a way that, at each step, one embeds a solid torus $T_i$ inside $T_{i-1}$, like $T_2$ is embedded inside $T_1$, obtaining an infinite  sequence: $T_1 \supset T_2 \supset \cdots T_i \supset \cdots$. The {\it Whitehead continuum} is the intersection of all the tori $T_i$, and the {\it Whitehead manifold} $Wh^3$ is the complement of it in the three-sphere $\mathbb S^3$. In a dual way, $Wh^3$ can be also described as being the ascending union of solid tori $T_j$ embedded inside $T_{j+1}$ so that $T_j$ is a Whitehead link in the interior of $T_{j+1}$ (and this is the classical description of $Wh^3 = \bigcup _1^{\infty} T_j$ ). See here Figure \ref{fig:1} and \cite{Po_4-dim, PT-Wh}.

\begin{figure}

\centering

\includegraphics{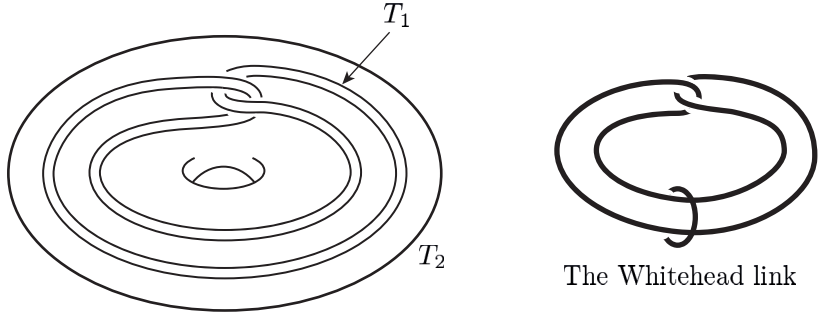}

\caption{The Whitehead manifold $Wh^3$ and the Whitehead link.} \label{fig:1}

\end{figure}

The main features of the Whitehead manifold are that $Wh^3$ is {\bf not} simply connected at infinity (intuitively, this means that there are loops at infinity that cannot be killed staying close to infinity) and that the product $Wh^3 \times \mathbb R$ is actually diffeomorphic to the standard $\mathbb R^4$ (this result was first showed by Shapiro and Glimm, and see here \cite{Glimm} where it is also proved that the Cartesian product of the Whitehead manifold with itself is topologically $\mathbb R^6$). Moreover, recently, D. Gabai has also demonstrate a surprising result: the Whitehead manifold can be described as the union of two embedded submanifolds, both homeomorphic to $\mathbb R^3$, and such that their intersection is also homeomorphic to $\mathbb R^3$ (see \cite{Gab_2}).

What is interesting for us here is that although $Wh^3$ is {\bf not} \textsc{gsc} (since this would imply $\pi_1^{\infty} Wh^3=0$, see the comment just after Lemma \ref{lemma5}), it still admits handlebody decompositions with geometric intersection matrix of the form difficult id + nilpotent, as outlined in \cite{Po_4-dim}.

For instance, consider the Figure \ref{fig:2}. There is there a disk whose boundary is a curve $T_1$, so that the disk intersects itself, and the intersection is a double line (denoted, in the picture, by $[b, c]$). This is a 2-dimensional object that leads to a handlebody
decomposition of $S^1 \times D^2$, whose geometric intersection matrix is $H_1^2 \cdot H_1^1 =1$ and $H_1^2 \cdot H_2^1 = 2$. More precisely this means that there is only one 0-handle, two 1-handles and one 2-handle, and the double line  $[b,c]$ of above is $H_2^1$. To go on, one takes now another closed loop which goes through both a piece of $T_1$ and  $H^1_2$, and next one deals with this loop like we have just done with the curve $T_1$, i.e. by adding along it another disk which hits itself back. Once again, this process leads to another handlebody
decomposition of $S^1 \times D^2$, whose geometric intersection matrix takes now the form $H_1^2 \cdot H_1^1 =1$,  $H_1^2 \cdot H_2^1 = 2$, and $H_2^2 \cdot  H^1_ 2 = 1$, $H^2_2 \cdot  H^1_3 = 2$.

\begin{figure}

\centering
\includegraphics[scale=0.6]{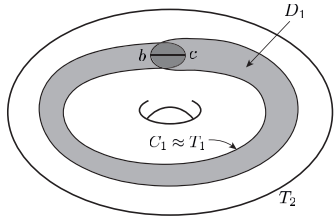}

\caption{Disks inside $Wh^3$.} \label{fig:2}

\end{figure}

This procedure can continue  indefinitely, getting an infinite matrix which is of the form
difficult id + nilpotent. But the problem is that the resulting object is not our manifold $Wh^3$ (since it is 2-dimensional and it is even not locally-finite).

However, one can cleverly manipulate the process in order to get the desired result. Firstly,  after each one
of the infinitely many steps, one may add  additional handles of index 1, 2 and 3, in order to obtain a 3-dimensional object. Then, it can be shown that, by paying close attention to details, one can actually manage to obtain a handlebody decomposition for
the Whitehead manifold which is {\sc proper} (i.e. it does not accumulate at finite distance), and whose geometric intersection matrix is made of a main
difficult id + nilpotent part, and of  some additional easy id + nilpotent parts
(which do not interfere with the main difficult part).

To conclude  our digression on this interesting manifold, we outline some of the developments that have occurred since the original input given by the appearance of $Wh^3$. First of all,  in the 60s, McMillan \cite{McMill} constructed uncountably many contractible open 3-manifolds with no two homeomorphic,  most of which are not universal covers of closed 3-manifolds, just because  there are only countably many  closed 3-manifolds and, therefore, only countably many open contractible
3-manifolds which cover closed 3-manifolds. But the first concrete examples of such manifolds (the so-called {\em genus-one Whitehead manifolds}) were constructed only afterwards in \cite{Mye}: these manifolds admit no non-trivial, free, and properly
discontinuous group actions, and thus they cannot cover non-trivially a non compact 3-manifold. Genus-one
Whitehead manifolds are irreducible, contractible, open manifolds which are monotone unions of solid tori but
are not homeomorphic to $\mathbb R^3$. The class of such examples was further enlarged by Wright in \cite{Wri}, where he constructed, for  each  $n \geq 3$,   specific  examples  of  contractible  $n$-manifolds  (called {\em Whitehead-type $n$-manifolds}) similar  to  Whitehead's  original  3-manifold (obtained, for instance, almost in the same way as $Wh^3$, just by considering solid $n$-tori instead of 2-dimensional ones)  which  cannot  non-trivially  cover  any  space. Finally, in \cite{FG}, Funar and Gadgil further studied Whitehead-type $n$-manifolds (that are open, contractible and not simply connected at infinity), and proved that there are uncountably many Whitehead-type $n$-manifolds for any  $n\geq 5$, and  that there are uncountably many such manifolds that are not geometrically simply connected!

\subsection{Dehn-exhaustibility and the QSF property}

We introduce next a notion which will be very useful for us here.

\begin{defi}\label{dehn}

A smooth open  $n$-manifold $M^n$ is said to be {\bf{Dehn-exhaustible}}
if for every  compact subset $k \subset M^n$ we can find a \textbf{compact} bounded $n$-manifold $K^n$ with $\pi_1 K^n=0$ , entering in the following
commutative diagram
$$\begin{diagram}
k &&\rTo^i(4,2)&&M^n \\
    &\rdTo _ j& &\ruTo_f\\& & K^n\end{diagram}  \e(1-5)$$
which is such that
\begin{itemize}
\item $i$ is the canonical inclusion and $j$ is an inclusion too,

\item $f$ is a smooth \textbf{immersion},
\item the following so-called {\bf{Dehn-condition}} is fulfilled:
$$ j(k) \cap M_2(f) = \emptyset ,\e(1-6)$$ where $M_2(f)$ is
the set of  points $x \in K^n$ such that ${\rm{card }} \{  f^{-1} f(x)\}> 1$.
\end{itemize}
\end{defi}

N.B. We will also need the related set of \textbf{double points} $M^2(f) \subset K^n \times K^n$, where $M^2(f)$ is the
set of points $ (x_1, x_2) \in K^n \times K^n $ such that $x_1 \neq x_2$ and $f(x_1) = f(x_2)$. This set comes with an obvious  projection on the first factor
$K^n \times K^n \supset M^2(f) \to M_2(f) \subset K^n$.

\vspace{3mm}

 There is quite some history behind this Definition  \ref{dehn}. In the nineteen-twenties, when topology was in its infancy, Max Dehn published a paper,
 a lemma of which had the following quite striking corollary: ``A knot $i: S^1 \hookrightarrow \mathbb S^3$ is unknotted if and only if $\pi _1 (\mathbb S^3
 - i(S^1)) = \mathbb Z$''. In the old days, this was called ``the fundamental theorem of knot theory'', and with Papakyriakopoulos' ``sphere theorem'', there is also an analogous result for links in $\mathbb S^3$.

 But few years later it happened that the proof of the lemma in question, called after that ``Dehn's lemma'', was shown to be irreparably wrong. Then, in the nineteen-fifties came the big breakthrough
 of Papakyriakopoulos, who not only proved Dehn's lemma \cite{Pap} but, in the same breath proved the sphere theorem. This opened the first mature period of 3-manifold topology before the Thurston revolution and its sequels appeared on the scene. One should also mention that a much more transparent version of the proof of Dehn's lemma was provided by Shapiro and Whitehead \cite{S-W}.

 But then, a bit later, while trying, in depth, to understand what was going on in the sphere theorem, John Stallings was lead to his own theorem:
 ``A torsion-free group with infinitely many ends is a free product'' \cite{St2, St_libro}, one of the big classical results in geometric group theory. And to finish this little historical prentice, Misha Gromov gave a very short crisp proof of Stallings theorem, using minimal surfaces \cite{Grom-ICM}.

Also, a few years after the appearance of \cite{Po2-duke} , Brick and
Mihalik \cite{BM1} abstracted the notion \textsc{qsf} (meaning
\textit{\textbf{quasi-simple filtration}}) from the earlier work
of Casson \cite{Ger-St} and of  the second author (V.P.)
\cite{Po2-duke, Po3-JDG}, running as follows:

\begin{defi}\label{qsf} A locally-compact simplicial complex $X$ is said to be \textsc{qsf} ({\bf{quasi-simply filtered}}) if for any compact
subcomplex $k \subset X$ there is a simply-connected compact
(abstract) complex $K$ endowed with an inclusion $k \overset {j}
{\longrightarrow } K$ and with a simplicial map $K \overset {f}
{\longrightarrow } X$ satisfying the Dehn-condition $M_2(f) \cap
j(k) = \emptyset$, and entering in a commutative diagram like (1 -- 5)
above (but now with a map $f$ which is no longer an immersion, but just a
simplicial map).
\end{defi}

In other words, a locally finite simplicial complex $X$ is
\textsc{qsf}  if it
satisfies something like the Dehn-exhaustibility, with the condition on $f$ relaxed from
immersion to $f$ being a mere simplicial map (see here \cite{BM1,
FO2, Ot_surv, St3}). Most of the good virtues of the Dehn-exhaustibility are preserved for the more general \textsc{qsf},
like for instance the implication $$V^3 \mbox { open, irreducible
and \textsc{qsf}} \Longrightarrow V^3 \simeq \mathbb R^3.\e(1-7)$$
But what one gains is very valuable, since, unlike Dehn-exhaustibility,
the \textsc{qsf} property turns out to be a group theoretical,
presentation-independent notion: if $K_1, K_2$ are two presentations
 (i.e. two presentation complexes) for the same finitely presented
  group $\Gamma$, then $\widetilde K_1 \in \mbox
{\textsc{qsf}} \iff \widetilde K_2 \in \mbox {\textsc{qsf}}$ \cite{BM1}.

\begin{defi}\label{qsf_groups}
If that happens, we say that the group $\Gamma$ itself
is \textsc{qsf}.
\end{defi}

Now, in order to better perceive all the previously defined notions, we list  several examples of spaces and groups which satisfy or do not satisfy these conditions. Of course, Euclidean spaces  are the easiest examples of spaces which are simply connected at infinity (in dimension $\geq 3$), \textsc{gsc}, \textsc{wgsc}, Dehn-exhaustible and \textsc{qsf}. On the other hand, as already said, the Whitehead manifold $Wh^3$, genus-one Whitehead manifolds  and Whitehead-type $n$-manifolds are examples of open manifolds which are not simply connected at infinity nor \textsc{gsc} nor  \textsc{qsf} (but they are \textsc{wgsc}). A funny example of a simplicial complex which is \textsc{qsf} but not \textsc{wgsc} is the universal cover of the 2-complex associated to the following presentation of $\mathbb Z = \langle a,b | baba^{-1} b^{-1}\rangle$ (see Figure \ref{fig:3}). In the 80s M. Davis \cite{Da} constructed the very first examples of (high-dimensional) compact aspherical manifolds (and finitely presented groups) whose universal covering spaces are not simply connected at infinity; on the other hand (most of) these compact manifolds  have hyperbolic or CAT(0) fundamental groups, and so Davis' examples are \textsc{wgsc} (and hence \textsc{qsf}). The striking result by the second author (V.P.), announced in \cite{Po_QSF_survey} and proved in \cite{Po_QSF1_Geom-Ded, Po_QSF2, Po_QSF3}, states that actually {\bf {all}} finitely presented groups are \textsc{qsf} (for easier, direct but partial results which concern just specific geometric classes  of finitely presented groups that are \textsc{qsf} see \cite{BM1, FO2, MT, St3}).

 \begin{figure}

\centering

\includegraphics{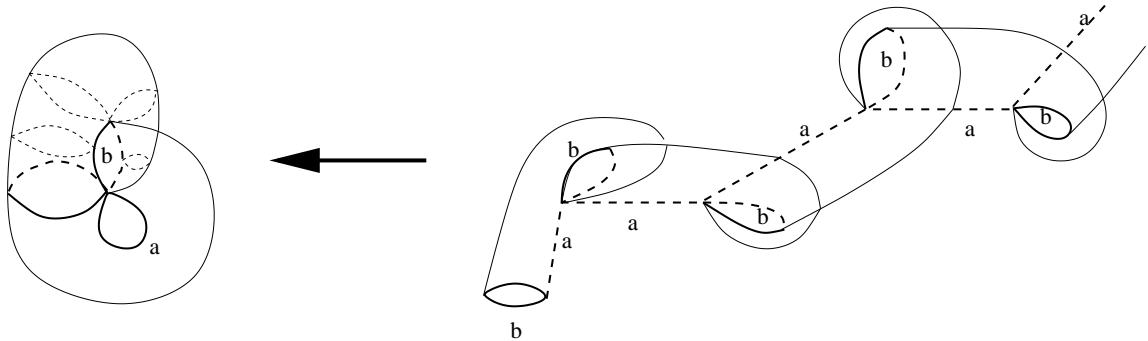}

\caption{A presentation 2-complex for  $\mathbb{Z}=\langle a,b | baba^{-1}b^{-1}\rangle$, and its  universal covering space, a \textsc{qsf} space that is \textbf{not} \textsc{wgsc}. In fact,  in the process of killing any loop $b$ of  the universal cover, one always creates another loop which can be killed by creating another loop, indefinitely.} \label{fig:3}

\end{figure}

\subsection{On  Dehn's lemma}

 So, we go back now to Dehn's lemma. In a modernised form, it can be stated as follows, and for  a proof see \cite{Po2-duke}.

 \begin{lem}[Dehn's lemma \`a la Po, \cite{Po2-duke}]\label{lemma4}
 Any open simply-connected 3-manifold $V^3$ which is Dehn-exhaustible, is also \textsc{wgsc}, and hence simply connected at infinity too.
 \end{lem}

 Note that the classical Dehn's lemma says that if an embedding $S^1 \subset M^3$ extends to a map $f: D^2 \to M^3$ which satisfies the Dehn-condition $S^1 \cap M_2(f) = \emptyset $, then $S^1$ is unknotted in $M^3$. Our Definition \ref{dehn} stems from these things.

 Next, here is how Dehn-exhaustibility and \textsc{gsc} come together. The result we will present now, as such, is unrelated to group theory. But then, on the one hand, things like it will be used later in this survey in order to prove group-theoretical results. On the other hand, it  should serve as an easy paradigm of how the
 $\Phi/\Psi$-manipulation can be used,  particularly in group theory.

 For the sake of the simplicity of exposition, we give here the next result only for the case of open smooth manifolds, although the definition of Dehn-exhaustibility can be extended to more general contexts, and then the lemma below too. And this extension will be useful.

 \begin{lem}[Stabilization Lemma]\label{lemma5}
 Let $M^n$ be  an open smooth manifold which is such that we can find a $p \in \mathbb Z_+$ with the property that $M^n \times B^p$ is \textsc{gsc}. Then $M^n$ is Dehn-exhaustible, with, in the context of diagram (1 -- 5), a smooth $K^n$ and a smooth immersion $f$.
 \end{lem}

 For $n=3$, the complete proof of Lemma \ref{lemma5} is given in \cite{Po2-duke}. But the arguments extend easily to Lemma \ref{lemma5} itself. We will present this below, after we open  a prentice.

Notice, first, that Lemma \ref{lemma5} implies that when it comes to the Whitehead manifold $Wh^3$, then for any $p \in \mathbb Z_+$ we {\textbf{cannot} have $Wh^3 \times B^p$ \textsc{gsc} because, in this case, combining Lemma \ref{lemma4} and Lemma \ref{lemma5} we would get that $\pi _1 ^{\infty} Wh^3 =0$, which is false. But then  $Wh^3$ itself cannot be \textsc{gsc} either.

Another comment is that we can extend Definition \ref{dehn} to a cell-complex context.

\begin{defi}[Definition \ref{dehn} extended]\label{extended}
Let $M^n$ be a simplicial complex which is {\bf{pure}}, i.e. each simplex $\sigma^{\lambda}$, with $\lambda <n$, is face of a simplex $\sigma ^n \subset M^n$. We say that $M^n$ is {\bf{Dehn-exhaustible}} if for every compact subcomplex $k \subset M^n$ there is a commutative diagram
$$\begin{diagram}
k &&\rTo^i(4,2)&&M^n \\
    &\rdTo _ j& &\ruTo_f\\& & K^n\end{diagram}\e(1-8)$$
where $K^n$ is a compact simply-connected \textbf{pure} complex, $f$ a simplicial \textbf{immersion} and the Dehn-condition $j(k) \cap M_2(f)=\emptyset$ is satisfied.
\end{defi}

 \begin{proof}[Proof of Lemma \ref{lemma5}.]
 (N.B. This proof, which is essentially the same as the argument in \cite{Po2-duke}, extends to the context of pure complexes and of Definition \ref{extended}.)

  When we consider $M^n \times B^p \supset M^n \times \{\ast\} \simeq M^n$ , with  $\{\ast \}=$ the center of $B^p$, then $M^n \times B^p$ can be endowed with a smooth triangulation, compatible with the DIFF structure, so that $M^n \times \{\ast\} \simeq M^n$ is a subcomplex. Also, because $M^n \times B^p$ is \textsc{gsc}, and hence \textsc{wgsc}, the simplicial complex $M^n \times B^p$ is exhausted by a union of compact simply-connected subcomplexes of codimension zero $Z_1 \subset Z_2 \subset \cdots \subset M^n \times B^p$. We move next to the $n$-skeleta $Z^n_1 \subset Z^n_2 \subset \cdots \subset \cup_i Z^n_i = Z^n_{\infty} =
  \{ \mbox {the $n$-skeleton of } M^n \times B^p\} \supset M^n \times \{\ast \}$ (which is a subcomplex of codimension zero).

After appropriate subdivisions and perturbations, like in \cite{Po2-duke, Po_QSF3}, one can ask that
$$ \pi_{\infty} = \pi |_{Z^n _{\infty}} : Z^n_{\infty} \to M^n\e(1-9)$$
be a simplicial non-degenerate map.

In the wake of (1 -- 9), we also define  $\pi_{i} = \pi |_{Z^n _{i}} $. For each $i\leq \infty$ we have the equivalence relations
$\Phi _i = \Phi (\pi_i)$ and $\Psi _i = \Psi (\pi_i)$. Here, for $i < \infty$, we clearly get $\Phi _{i} = \Phi_{\infty} |_{Z^n_i}$, but, generally speaking, we only have an inclusion $\Psi _{i} \subset  \Psi_{\infty} |_{Z^n_i}$.

For our triangulated $M^n$ we have the following commutative diagram of simplicial maps, with $\overline{\pi}_{\infty}$ the analogous of $f_1$ from (1 -- 2):

$$\begin{diagram}
Z^n_{\infty}/ \Psi_{\infty} &&\rTo^{ \overline{\pi}_{\infty}}(4,2)&&&M^n \\
    &\luTo_ i& & &    \ruTo_{\simeq}& \\
    & &  M^n / \Psi_{\infty} &= M^n \times \{ \ast \}&\end{diagram}$$
In other words, the projection $\overline{\pi}_{\infty}$ admits a \textbf{cross-section}.

Notice that the map $i$ \textbf{has} to be bijective, since otherwise $\overline{\pi}_{\infty}$ would have singularities (i.e. it would fail to be immersive). This means that $M_2 (\overline{\pi}_{\infty}) = \emptyset$ (see Definition \ref{dehn}). With this, in the following commutative diagram
\begin{diagram}
Z^n_{\infty}         &\rTo   & Z^n_{\infty}  / \Phi_{\infty}\\
\dTo  &    \rdTo^{\pi_{\infty}\ \ \ \ \ \quad \quad}    \ruTo   &\dTo_{\simeq}\\
  Z^n_{\infty}  / \Psi_{\infty}       &\rTo_{\simeq}   & M^n
\end{diagram}
both horizontal and vertical maps into $M^n$ are bijective and hence, essentially because there is a cross-section, we find that the following two equivalence relations have to coincide
$$\Phi _{\infty} = \Psi_{\infty}.$$
There is then, starting from the equality $\Phi _{\infty} = \Psi_{\infty}$, an easy ``compactness argument" (see \cite{Po2-duke}) which shows that there exists a function among positive integers $\mathcal N: \mathbb Z_+\to \mathbb Z_+$ such that $\mathcal N(j)>j$, with the following property: $\Psi _{\mathcal N(j)} |_{Z^n_{j}} = \Phi _j$ for all $j$'s.

Fix now a compact subset $k \subset M^n$. There is an $i < \infty$ such that $\pi_i (Z_i ^n )\supset k$, and since $\pi_{\infty}^{-1} (\pi_i (Z^n_i))$ is compact too, we find some $j >i$ such that $Z_j^n \supset \pi_{\infty}^{-1} (\pi_i (Z^n_i)) \supset \pi_{\infty}^{-1} (k)$. Our set-up is such that, whenever we have $(x_1, x_2) \in M^2(\pi_{\infty})$ and $x_1 \in k$, then $x_2 \in Z^n_j$. For our $j$ we go now to the $\mathcal N(j) > j$ introduced above, coming with $\Psi _{\mathcal N(j)} |_{Z^n_{j}} = \Phi _j$, and to the following sequence of maps and inclusions:
$$k \subset Z^n_j / \Phi _j = Z^n_j / \Psi _\mathcal N \subset Z^n_{\mathcal N}/ \Psi _{\mathcal N} \overset{\overline{\pi}_{\mathcal N}}{\longrightarrow} M^n,$$
with  $\overline{\pi}_{\mathcal N}$ like the $f_1$ in (1 -- 2). Here $Z^n_{\mathcal N}/ \Psi _{\mathcal N}$ is a compact $n$-manifold which, by Lemma \ref{lemma2}, is simply connected, and $\overline{\pi}_{\mathcal N}$ is immersive. We have also that $\pi_{\infty}^{-1} (k) \subset Z^n_j$ and $\Psi _{\mathcal N} |_{Z ^n_j}= \Phi _j$.

This means that $k$ \textbf{cannot} be involved with the double points of $\overline{\pi}_{\mathcal N}$, i.e. $k \cap M_2 (\overline{\pi}_{\mathcal N}) = \emptyset$ (Dehn-condition). With this, the commutative diagram
$$\begin{diagram}
k &&\rTo^i(4,2)&& M^n \\
    &\rdTo_j& &\ruTo_{\overline{\pi}_{\mathcal N}}\\& &Z^n_{\mathcal N} / \Psi _{\mathcal N}
\end{diagram}$$
with $i,j$ the obvious inclusions, has all the features of (1 -- 5) in Definition \ref{dehn}. This ends the proof.
 \end{proof}
\begin{rem}
 The argument above is a paradigm for many such, occurring in our group theoretical context. And then the ``compactness argument''
 will occur again, coming with the function $\mathcal N: \mathbb Z_+\to \mathbb Z_+$. It is possibly worthwhile to try to connect asymptotic properties of the function $\mathcal N$, like growth, with the asymptotic properties of the corresponding finitely presented group $G$.
\end{rem}

 \subsection{Presentations and REPRESENTATIONS of groups}

 A combinatorial presentation of a finitely presented group $G$ is a given finite system of generators $g_i^{\pm 1}$ and relators $\mathcal R_j$ for $G$. Traditionally, to this combinatorial presentation, one attaches a geometric presentation. This is usually a finite 2-complex $K^2$ with one vertex and with 1-cells and 2-cells corresponding, respectively, to the $g_1 ^{\pm 1}, g_2 ^{\pm 1}, \cdots, g_n ^{\pm 1}$ and $\mathcal R_1 ^{\pm 1}, \mathcal R_2 ^{\pm 1}, \cdots, \mathcal R_p ^{\pm 1}$. Classically, the 1-skeleton of $\widetilde K^2$ is the Cayley graph of $G$.

 In our present survey we will deal, mostly, with completely general finitely presented groups $G$, but we will need to be quite choosy when it comes to geometric presentations. Such {\textbf{presentations}} will always be, for us, 3-dimensional finite complexes $M^3(G)$ such that $\pi_1 M^3(G) = G$. And since we want to be able to deal with any finitely presented group $G$, the $M^3(G)$ has to have \textbf{singularities}, by which we mean non-manifold points. Actually, $M^3(G)$ will be a compact \textbf{singular 3-manifold} with a very precise kind of singularities, the so-called ``\textbf{\textit{undrawable singularities}}}'' which were a lot used by the second author in earlier work, see here \cite{Gab, Ot-Po1, Po4-Top1}.  They are displayed as little shaded rectangles in Figure \ref{fig:4}.

 \begin{figure}

\centering
\includegraphics[scale=0.9]{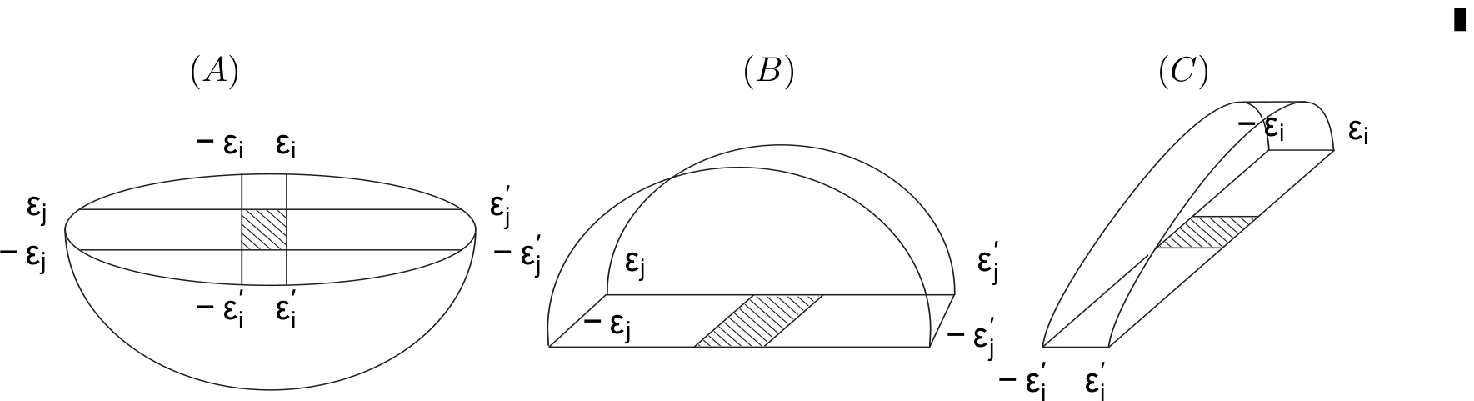}

\caption{The $N^3(S) \subset M_0^3(G)$ is gotten by putting these three 3-dimensional pieces together. The rectangles $[ \epsilon_j, -\epsilon_j; \epsilon_j', -\epsilon_j' ]$, $[ \epsilon_i, -\epsilon_i; \epsilon_i', -\epsilon_i']$ in (B), (C), respectively, are glued to the ones in (A). The immortal singularity $S$ is shaded. The $\partial N^3(S) = S^1 \times S^1 -\rm{int} D^2$ is made up from three contributions: the half-sphere in (A) and the two curved arches occurring in (B) and (C).} \label{fig:4}

\end{figure}

 Another way to describe them is the following. Start with $R_1 ^2 = (x=0, y, z), R_2^2 = (x, y=0, z) \subset \mathbb R^3 =(x,y,z)$ and with their canonical half-line $\frac{1}{2} L = (x=0, y=0, z\leq 0) \subset R_1^2 \pitchfork  R_2^2$. We consider then the obvious map $$f: K^2 = R_1^2 \underset{\frac{1}{2}L}{\cup} R_2^2 \longrightarrow  \mathbb R^3, \e(1-10)$$ which is immersive, except at its singular point $s = (0,0,0)$. This is a 2-dimensional undrawable singularity and its 3-dimensional thickening is the kind of singularity $M^3(G)$ possesses.

 Here is a typical example of such a 3-dimensional presentation $M^3(G)$ of $G$. Start with a combinatorial presentation, by generators and relators of $G$, and to it we will attach a particular 3-dimensional geometric presentation which we will call $M_0^3(G)$.

 Corresponding to the generators $g_1 ^{\pm 1}, g_2 ^{\pm 1}, \cdots, g_n ^{\pm 1}$ we pick up a 3-dimensional handlebody of genus $n$, i.e. a smooth bretzel $T^3$ with $n$-holes. To the relations corresponds a generic immersion
$$ \sum_{i=1} ^r S_i^1
\overset {\Upsilon} {\longrightarrow} \partial T^3, \e(1-11)$$
 which we immediately thicken into
$$ \sum_{i=1} ^r (S_i^1 \times [-\epsilon _i,\epsilon _i])
\overset {\Upsilon} {\longrightarrow} \partial T^3. \e(1-12)$$

We can actually always slightly change the original combinatorial presentation so that for each $i$, the individual maps $\Upsilon |_{S^1_i}$ and $\Upsilon |_{S^1_i\times [-\epsilon _i,\epsilon _i] }$ actually \textbf{inject}. This will be assumed from now on (and it will make life easier).

Finally, for each relator $\mathcal R_i$ we attach to $T^3$ a 2-handle $D_i^2\times [-\epsilon _i,\epsilon _i]$, of attaching zone $\partial D^2_i\times [-\epsilon _i,\epsilon _i]= S_i^1 \times [-\epsilon _i,\epsilon _i]$, along $\Upsilon |_{S^1_i\times [-\epsilon _i,\epsilon _i] }$. This final object is our $M_0 ^3 (G)$, and this will be the typical example of our geometric \textbf{\textit{presentations of groups}}.

This object is singular and, unless $G$ is the fundamental group of a smooth 3-manifold, a very exceptional case indeed, singularities have to be there. This is how they occur. Any transversal contact $$ \Upsilon (S^1_{i_1} ) \pitchfork \Upsilon (S^1_{i_2}) \ni x \in \Upsilon (M_2 (\Upsilon)), i_1 \neq i_2,$$
 gives rise to connected components of $\Upsilon (M_2 (\Upsilon))$ (with $\Upsilon$ like in (1 -- 12)), which are little squares $ [-\epsilon _{i_1},\epsilon _{i_1}] \times [-\epsilon _{i_2},\epsilon _{i_2}] \subset \Upsilon (S^1_{i_1} \times [-\epsilon _{i_1},\epsilon _{i_1}]) \cap \Upsilon (S^1_{i_2} \times [-\epsilon _{i_2},\epsilon _{i_2}])$. These non-manifold regions for $M_0^3(G)$, which certainly are 3-dimensional singularities of the undrawable type mentioned above, will also be called \textbf{\textit{immortal singularities}}, so as to distinguish them from the singularities occurring in the $\Phi / \Psi$  little theory, that are non-immersive points at the source, and which we will call \textbf{\textit{mortal singularities}}, since these singularities do get killed by the $\Psi \subset \Phi$ equivalence relation (i.e. by the zipping).

Let us be more precise at this point. The \textit{immortal singularities} are points
$x$ of some $n$-dimensional space $X$, where $X$ fails to be an $n$-dimensional manifold, locally; while
\textit{mortal singularities} are points $x\in X$ where a non-degenerate map $f: X \to Y$
fails to be immersive, locally again.

Coming back to the presentation $M_0^3(G)$, if by any means $G$ is the fundamental group of a closed (or just compact) 3-manifold $M^3$, then we will take as $M_0^3(G)$  the smooth $M^3$ itself.

The $M_0^3(G)$  described above is an instance of what we will call a \textbf{\textit{singular handlebody}} $P^n$ (for $n \geq 3$). Such an object is, by definition, a union of $n$-dimensional handles of index $\lambda = 0, 1, 2, \cdots, n$, call them generically $h^{\lambda} = B^{\lambda} \times B ^{n-\lambda}$. They have attaching zones $\partial _a h ^{\lambda} \subsetneq \partial h ^{\lambda}$, defined by $\partial _a h^{\lambda} = \partial B ^{\lambda} \times B ^{n-\lambda}$ and lateral surfaces $\delta h ^{\lambda} = B^{\lambda} \times \partial B ^{n-\lambda}$.  These handles $h^{\lambda}$ are put together into $P^n$, according to the following rules of the game:
\begin{enumerate}
\item as usual, $\partial _a h ^{\lambda}$ is glued to the free part of $\underset{\mu < \lambda}{\cup} \partial h ^{\mu}$;
\item but we allow now for non-trivial intersections $\partial _a h ^{\lambda}_i \cap \partial _a h ^{\lambda}_j $, for $i\neq j$ and $\lambda \geq 1$.
\end{enumerate}

\noindent This creates singularities, of a more general type than the undrawable ones.

The $M_0^3(G)$ are only one class of group-presentations, but, for technical reasons, another class, the ``\textbf{good-presentations}'', will soon have to be introduced too. There are good reasons, as we shall see, to use our present 3-dimensional geometric presentations of $G$, rather than the classical 2-dimensional presentations.

We are finally ready to introduce the \textsc{representations} of $G$ (for convenience, we will skip from now on the terms ``topological" and ``inverse") . The $G$ will invariably be for us a finitely presented group of the most general type.

\begin{defi}\label{repr}
Start with an arbitrary geometric presentation of a finitely presented group $G$, $M^3(G)$ and go then to its universal covering space $\widetilde {M^3(G)}$.
A \textsc{representation} of  $G$ is a \textbf{non-degenerate} simplicial
map $$ X^p \overset {f} {\longrightarrow} \widetilde {M^3(G)} \e(1-13)$$
where:
\begin{itemize}

\item[(a)] the space $X^p$ is a cell-complex of dimension $p=2$ or $3$  which is \textsc{not} necessarily locally finite, but which is  geometrically simply connected ({\sc gsc}). [Notice that since the map $f$ is non-degenerate, we must have $p \leq 3$; but the case $p=1$ is of absolutely no interest. According to the cases $p=2$ or $p=3$, we will talk about $2$-dimensional or $3$-dimensional \textsc{representations}, and both are important for us].

\item[(b)] $\Psi(f)= \Phi(f)$.

\item[(c)] The map $f$  is ``\textbf{essentially surjective}'',
which means the following. If $p=3$, then for the closure of the image we have
$\overline{f(X^3)}= \widetilde {M^3(G)})$; while if $p=2$, it means that  $\widetilde {M^3(G)}$  can be gotten from $\overline{f(X^3)}$ by adding cells of dimensions $\lambda =2$ and $\lambda =3$. \end{itemize}
\end{defi}

Note that, in the context of a 2-dimensional \textsc{representations} of $G$, $f: X^2 \to \widetilde {M^3(G)}$,
 we may have immortal singularities of $\widetilde {M^3(G)}$ and, for $X^2$, both
 mortal and immortal singularities $x \in X^2$. In this last case, in the context of the
 present paper, the following things will always happen.

 \begin{itemize}
 \item At a singular $x \in X^2$, whether immortal or mortal, locally at $x$, our $X^2$ is like
  in (1 -- 10), namely of the form $R_1^2 \underset{\frac{1}{2}L}{\cup} R_2^2$ which we call an ``undrawable singularity'', for obvious reasons. (This is a terminology introduced long ago by Barry Mazur, during some lengthly discussions
  with the second author).

 \item Locally, at any immortal singularity $x \in X^2$, the map $f$ embeds and we have
 $f(x) \in \{$some immortal singularity of $\widetilde {M^3(G)} \}$.

 \end{itemize}

\begin{rem}\
\begin{itemize}
\item  According to Misha Gromov, groups are geometric objects defined up to quasi-isometry \cite{Grom-ICM, Grom-AsInv}. From this viewpoint, $G$ and $\widetilde {M^3(G)}$ are the same object, certainly the same object up to quasi-isometry. And with this, let us notice that while mundane group representations are arrows of the form $\{G \to \cdots \}$, going homomorphically into some other group, our \textsc{representations} take just the dual form $\{\cdots \to G\}$. And actually this is why they are called ``inverse-\textsc{representations}''. But, as a matter of taste, we prefer the capital letters to the adjective ``inverse''.

    \item We will also introduce a slightly weaker notion, the \textsc{wgsc-representations}. One proceeds exactly like in Definition \ref{repr}, but simply replaces the \textsc{gsc} condition by \textsc{wgsc}.

    \item In very simple words a \textsc{representation} is a sort of resolution of the universal cover of a presentation of $G$ in the realm of cell complexes which are \textsc{gsc} (or \textsc{wgsc}).

\item The reader may also have noticed that the group structure does not play any role in Definition  \ref{repr}, as such. But then our definition opens the door for the following possibility, which will be all-important for the more serious applications. We may have \textbf{equivariant} \textsc{representations}, where, besides the obvious action $G \times \widetilde {M^3(G)}\to \widetilde {M^3(G)}$, we also have a second free action $G \times X^p \to X^p$ such that for each $x \in X^p$ and $g \in G$, we should have the equivariance condition: $f (gx) = g f(x)$.

Such equivariant \textsc{representations} are very useful, but they are hard to get. We will come back to them soon.

\item But, once the group structure of $G$ does not play any role in Definition \ref{repr}, we may also \textsc{represent} other 3-dimensional objects, not connected to any symmetry group  $$X^p \longrightarrow Y^3$$ and then $Y^3$ may or may not be a 3-manifold. Anyway, the existence of the \textsc{representation} above, forces that we should necessarily have (and see here Lemma \ref{lemma2}) $\pi_1 Y^3 =0$.

In this context, in \cite{PT-Wh} it is shown that when we \textsc{represent} the Whitehead manifold $Wh^3$, then quite naturally the Julia sets from the dynamics of quadratic maps pop up, raising of  course the issue of a possible connection between wild topology and dynamical systems (for more details see the last section).

\end{itemize}
\end{rem}

\subsection{Good-presentations of finitely presented groups}
We introduce now the \textbf{\textit{good-presentations}} $M^3(G)$ of $G$. These will have the very useful feature that not only are $M^3(G)$, $\widetilde {M^3(G)}$ singular handlebodies, but they also admit \textsc{representations} $X^3 \to \widetilde {M^3(G)}$, where $X^3$ itself is such a singular handlebody, and hence the non degenerate map $f$ sends, isomorphically, $\lambda$-handles into $\lambda$-handles, a very useful feature indeed.

To construct such good-presentations $M^3(G)$ we start with our previous $M_0^3(G)$, which was directly related to a given combinatorial presentation of $G$.

Let $S$ be any individual immortal singularity of $M^3(G)$, a little shaded square like in Figure \ref{fig:4}. This $S$ has a canonical neighbourhood $N^3(S) \subset M^3(G)$, the detailed structure of which is shown in Figure \ref{fig:4} and which is split from the rest of $M^3_0(G)$ by a surface of genus one, namely by $$\partial N^3(S) = S^1 \times S^1 - \overset{\circ}{D^2}.$$
 [N.B.: By $\partial N^3(S) $ we mean here not the whole boundary, but a piece of it, the {\textbf{\textit{splitting locus}}}].

We will break now $M^3_0(G)$ into a union of smaller pieces than the initial handles, but which are subcomplexes of $M^3_0(G)$ $$M^3_0(G) = \bigcup_1^q \Delta _i^3 (0).\e{(1-14)}$$
 These $\Delta^3_i(0)$'s are supposed to have disjoined interiors. For each singularity $S$, the $N^3(S) $ itself is a $\Delta^3_i(0)$. All the others are copies of $B^3$, contained now inside a given handle, each. These smooth $\Delta^3_i(0)$'s come with $\rm {int} \Delta^3_i(0) \cap \partial (\mbox{handle}) = \emptyset$, something which the singular $N^3(S)$'s violate.

 The important thing is that, in the context of (1 -- 14) the following condition is fulfilled: whenever $\rm{dim} (\Delta^3_i(0) \cap \Delta^3_j(0) ) > 1$, then we also have $$ \partial \Delta^3_i(0) \cap \partial \Delta^3_j(0) = \{ \mbox{a \textbf{unique} 2-cell } \xi\}, \mbox{called a \textbf{\textit{face}}}.$$

 This is what we have gained when we refined the singular handlebody  decomposition of $M^3_0(G)$, in its original description, into (1 -- 14). We proceed next via the following steps:
 \begin{itemize}

\item[ a)] We triangulate, compatibly, the $\Delta^3_i(0)$'s and their common faces $\xi$. From this, we retain now only the 2-skeleta $(\Delta^3_i(0))^{(2)}$, $\xi^{(2)}$.

\item [ b)] Each of the 2-dimensional complexes $(\Delta^3_i(0))^{(2)}$, $\xi^{(2)}$ is changed back into a singular handlebody, via the standard procedure below:
$$\{\lambda \mbox{-dimensional simplex}\} \to \{ 3\mbox{-dimensional handles of index } \lambda \}.$$

The $(\Delta^3_i(0))^{(2)}$, $\xi^{(2)}$ become now 3-dimensional singular handlebodies denoted respectively by $\Delta ^3_i$ and $g^3$. We may assume  each $\Delta ^3_i$, $g^3$ to be \textsc{gsc} and, when this is the case, $g^3 \subset \Delta^3_i$ is now ``sub-handlebody'', the meaning of this term should be obvious.

\item[ c)] We can glue now together the $\Delta ^3_i$, $g^3$'s following the prescription given by (1 -- 14). The result is, by definition, our good presentation of $M^3(G)$ (or one of such).
\end{itemize}

\noindent When the opposite is not explicitly said, our group presentations will always be good, in the text which follows.

\subsection{Some historical remarks}
\textsc{Representations} of topological spaces were
already present at an intuitive level of the initial step in the
approach of the second author to Poincar\'e Conjecture (see \cite{Gab} and, for a
complete up-to-date overview, see \cite{Po_pre0}).
In the papers \cite{Po4-Top1, PT4-AMS} homotopy 3-spheres $\Sigma
^3$ and open 3-manifolds were represented. It is actually in \cite{Po4-Top1},
under a different name, that \textsc{representations} first
appeared (see also \cite{Po-contMath}).

Without giving all  details, let us state the first ``representation-result'' from \cite{Po4-Top1} (the so-called ``Collapsible Pseudo-Spine Representation Theorem''):

\begin{thm} [V. Po\'enaru, \cite{Po4-Top1}]\label{TOP_repr1}

Given a homotopy $3$-sphere $\Sigma^3$, one can construct a
\textsc{representation} $f\colon K^2\to\Sigma^3$,   where $K^2$ is a finite $2$-complex and
$f$ a non-degenerate simplicial map with
undrawable singularities, like in (1 -- 10), such that  the complement of $f(K^2)$ is a
finite collection of open $3$-cells,  $K^2$ is collapsible and
$\Psi (f) = \Phi (f)$.
\end{thm}

Afterwards, trying to adapt this kind of result to open 3-manifolds,
Po\'enaru and Tanasi   \cite{PT4-AMS} gave an extension of these
ideas to the case of  simply-connected open $3$-manifolds $V^3$.
The main point here is that, at the source of the \textsc{representation}, the
set of double points of $f$ is, generally speaking, no longer
closed (unlike as it was in Theorem \ref{TOP_repr1}). One should have in mind that, whenever this set is \textbf{closed}, then things get
very much easier. For instance, in such a special case, the open manifold $V^3$ turns out to be even
simply connected at infinity (more precisely \textsc{qsf}, and
hence \textsc{wgsc} and \textsc{sci} \cite{Po-contMath, PT4-AMS}).

On the other hand, when one extends to the classical Whitehead 3-manifold $Wh^3$
the Collapsible Pseudo-Spine Representation Theorem (Theorem \ref{TOP_repr1}), new features occur. Actually, in \cite{PT-Wh}, Po\'enaru and Tanasi
studied the  most natural
2-dimensional inverse-\textsc{representations} of $Wh^3$, $g: X^2
{\longrightarrow } Wh^3$,
and they found that, not only the set of double points $M_2(g)$ is
\textbf{not} a closed subset, but its accumulation pattern (with accumulation on a Cantor set) is
chaotic. Very explicitly, the pattern in question is guided by a
specific class of \textit{Julia sets} generated by the infinite iteration
of real quadratic polynomials. More about this in the last section.

\subsection{Universal covers}
We present now an easy, standard procedure for getting \textsc{representations} $f: X^3 \to \widetilde {M^3(G)}$ and, at the same time, we will develop a ``\textbf{naive theory of the universal covering spaces}'', following \cite{Po4-Top1} (see also \cite{Po-Ta_Geom-Ded}).

So, we are ready now for constructing a \textsc{representation} $$ f: X^3 \longrightarrow \widetilde {M^3(G)}. \e{(1-15)}$$
 Pick up among the $q$ singular handlebodies  $\Delta _i ^3$ from the end of Section 2.7, a $\Delta_{i_1} ^3$. We can certainly find then a $\Delta_{i_2} ^3$ coming with a common face $g^3(1) =  \Delta_{i_1} ^3 \cap \Delta_{i_2} ^3$. This process continues until we have managed to put up an \textit{arborescent} union $\Delta ^3$ of all the $\Delta_{i} ^3$'s, this will be a \textbf{fundamental domain} for our $\widetilde M^3(G)$.
 Let $\Sigma = (\overline{g_1}, \overline{g_2}, \cdots, \overline{g_{2N}} )$ be the set of those faces \textsc{not} used for constructing $\Delta ^3$ and which are now free faces of $\Delta^3$. Careful here, the $\overline{g_i}$'s are themselves 3-dimensional handlebodies, like $\Delta ^3$, and we will do not write things like ``$\overline{g_i} \subset \partial \Delta ^3_i$'', which do not make sense.
 Since each $g^3 = \overline{g_i}$ is common exactly to two $\Delta^3 _k$'s, we have an obvious \textbf{fixed point free involution} $j: \Sigma \to \Sigma$. This $j$ is such that $M^3(G)$ is gotten from $\Delta ^3$ by identifying each $\overline{g_i}$ to $j \overline{g_i}$. One should also notice here that $${\rm{Sing}} \big(M^3 (G)\big) \subset \Delta^3 - \bigcup _{i=1}^{2N} \overline{g_i}.$$
  Let $\overline{G}$ be the free monoid generated by the abstract symbols $1, \overline{g_1}, \overline{g_2}, \cdots, \overline{g_{2N}}$. If we add the relations $ \overline{g_i} \cdot j \overline{g_i} = j \overline{g_i} \cdot  \overline{g_i} =1$, then $\overline G$ becomes a free group with $j \overline{g} =  {\overline{g}}^{-1}$. This also comes with an obvious surjective homomorphism $$\chi:  \overline G \rightarrow G .\e{(1-15-1)}$$
   We are ready now to build a representation space $X^3$. This will be $$ X^3 = \{ \mbox{The quotient space of the arborescent union } \underset{x \in \overline G}{\Sigma} x \Delta ^3\}, \e{(1-16)}$$
 where, in a Cayley graph manner, for each $x \in \overline G$ and each $ \overline{g} \in \Sigma$, we glue together
 $x\Delta ^3$ and $x \overline{g} \Delta^3$ along their two faces below
 $$ x\Delta^3 \supset  \overline{g} \leftrightarrow j  \overline{g} \subset x   \overline{g} \Delta^3.$$

 What our (1 -- 16) defines, is a non locally finite 3-dimensional singular handlebody which we will endow with the weak topology. This $X^3$, which being an arborescent union of \textsc{gsc} pieces is automatically \textsc{gsc} itself, also comes endowed with a tautological map which is non-degenerate ($\lambda$-handles go isomorphically to $\lambda$-handles) $$F: X^3 \longrightarrow M^3(G).\e{(1-17)}$$
 Notice also that the map $F$ ``unrolls'' the infinitely many fundamental domains of $X^3$ onto the unique fundamental domain of which $M^3(G)$  is the quotient, in a way which is reminiscent of the ``developing map'' of Sullivan and Thurston \cite{Su-Thu}.
 \begin{lem}[The naive, Kindergarten theory of universal covering spaces]\label{lemma6}
  When, like in (1 -- 2) we go to the commutative diagram
$$\begin{diagram}
X^3 &&\rTo^F(4,2)&&M^3(G) \\
    &\rdTo& &\ruTo_{F_1}\\& & X^3 / \Psi (F) , \end{diagram} $$
then the map $F_1 : X^3 / \Psi (F) \to M^3(G)$ \textbf{is} the universal covering space
  $\pi: \widetilde {M^3(G)} \to M^3(G)$.
 \end{lem}

 \begin{proof}
 This is immediate: $F_1$ clearly has the path lifting property and $\pi _1 (X^3 / \Psi (F))=0$, because of Lemma \ref{lemma2}. \end{proof}

 We call the present theory naive and ``Kindergarten'', with the idea that a child, with enough cubes available could play around
 with them, gluing then together and spreading them out (in higher dimension) and discover this way the universal covering space.

 Consider now the natural tessellation
 $$\widetilde {M^3(G)} = \sum _{g\in G} g \Delta ^3.$$
  Any initial lift of $1\cdot \Delta^3 \subset X^3$ to $\widetilde {M^3(G)}$ comes automatically with a canonical lift $f$ of the whole $X^3$ to $\widetilde {M^3(G)}$ and this map $f$, occurring in the commutative diagram
 $$\begin{diagram}
X^3 &&\rTo^f(4,2)&&\widetilde {M^3(G)}\\
    &\rdTo_F& &\ldTo_{\pi}\\& & M^3 (G) ,\end{diagram}\e{(1-18)}$$
  is non-degenerate. In the context of (1 -- 18) one has the following equality between equivalence relations
  $$\Psi (f) = \Psi (F) \subset X ^3\times X^3.$$
   This follows from the fact that $\pi$ is a covering projection and that the singularities of the maps $f$ and $F$ have to be killed by the same folding maps. But then, also, clearly $X^3 / \Psi(f)= \widetilde {M^3(G)}$ by the Kindergarten Lemma \ref{lemma6}. With this, we can find a cross-section for the canonical map $ f_1(\Psi(f)) : X^3 / \Psi(f) \to \widetilde{M^3(G)}$, with $ f_1(\Psi(f))$ like in (1 -- 2), namely $f_1(\Psi(f)) : \widetilde {M^3(G)} =
   X^3 / \Psi(F)= X^3 / \Psi(f) \to \widetilde {M^3(G)}$, where the first equality comes from Lemma \ref{lemma6}, and hence $\Psi(f) = \Phi(f)$.

   With these things we have the:
   \begin{lem}\label{lemma7}
The non-degenerate map among singular handlebodies, constructed above
$$f: X^3 \longrightarrow \widetilde {M^3(G)},\e{(1-18-1)}$$
is a \textsc{representation} of the group $G$.
\end{lem}

\begin{proof}
In order  to get some intuitive feeling for this result we suggest to the reader playing with the toy-model where $M^3(G)$ is replaced by the 2-torus $S^1 \times S^1$. The representation space (playing the role of $X^3$) is then an infinitely ramified union of Riemannan surfaces of $\log z$, each of them containing a non locally finite infinite spiral, put together along an infinitely ramified arborescent structure.

There is a natural left-action $\overline G \times X^3 \to X^3$ which, via (1--15--1),  connects to the natural action $G \times \widetilde {M^3(G)}\to \widetilde {M^3(G)}$. We have a commutative diagram, for any $\overline g \in \overline G$,
$$\begin{diagram}
  X^3& \rTo ^{\overline g \in \overline G} & X^3\\
\dTo_f &&  \dTo_f\\
\widetilde {M^3(G)}& \rTo ^{\ \chi (\overline g) \in G}& \widetilde{M^3(G)}.
\end{diagram} \e (1-19)$$
\end{proof}
\begin{rem}
The very elementary and simple-minded (1 -- 19) should not be mixed with the much more sophisticated, and much harder to achieve, \textbf{equivariant} \textsc{representations}, where we have an action $G \times X^3 \to X^3$, compatible with the action of $G$ on $\widetilde {M^3(G)}$, and not just an action of $\overline G$ on $X^3$. \end{rem}

Our $X^3$ is an infinite union of continuous paths of fundamental domains $\overline g \Delta ^3$, all starting at $1\cdot \Delta ^3$. For any $M < \infty$ the restriction to the paths of length $\leq M$ is denoted by $X^3 | M \subset X^3$. For any $\overline g \in \overline G$, the $\overline g (X^3 | M) \subset X^3$ is an isomorphic copy of $X^3 | M$ starting at $\overline g \Delta ^3$.

We introduce the notation $$ \Phi (M) = \Phi \big (f |_{ (X^3 | M)}\big ), \ \ \Psi (M) = \Psi \big (f |_{(X^3 | M)}\big ).$$
 The fact that $\Phi(f)= \Psi (f)$ has the following consequence, very much like the ``compactness argument'' mentioned during the proof of Lemma \ref{lemma5}.

 \begin{lem}\label{lemma8}
 Given our \textsc{representation} (1--18--1) there is a function $\overline M : \mathbb Z^+ \to \mathbb Z^+$, with $\overline M >>M$, analogous to the function $\mathcal N: \mathbb Z^+ \to \mathbb Z^+$ occurring in the completely different, non group-theoretical context of Lemma \ref{lemma5}, which is such that
 $$ \mbox{for all $M$, with $\overline M = \overline M (M)$, we find: } \Psi (\overline M) | (X^3 |M) = \Phi (M) | (X^3 |M).\e{(1-20)}$$
 \end{lem}

 This finishes what we have to say here concerning the highly pathological \textsc{representation} of $G$ produced by Lemma \ref{lemma7}, which we will hardly ever   be used, as such.

\subsection{Easy REPRESENTATIONS}

 All the material presented above was the preliminary easiest  part of the theory of \textsc{representations}. But before we will go into more difficult  results, we would like to show that even this easy part has already some interesting consequences.

 Anyway, by now we have established that, for any finitely presented group $G$, 3-dimensional \textsc{representations}
 do exist, but the one given by (1--18--1) and/or Lemma \ref{lemma7} is not very useful. And then, from 3-dimensional \textsc{representations} for $G$, 2-dimensional \textsc{representations} can be gotten too, by first finely subdividing and then going to an appropriate generic perturbation of the map $f | \{ \mbox{2-skeleton of the subdivided } X^3\}$ (we will give more details at the end of this section). There is an inverse (more difficult) road too, from 2-dimensional  to 3-dimensional \textsc{representations}, but we do not need to go into that right now.

\begin{defi}\label{easy}\
\begin{itemize}
\item[(I)] A 3-dimensional \textsc{representation} $$f: X^3 =  \sum _{\overline g\in \overline G} \overline g \Delta ^3  \longrightarrow \sum _{g\in G} g \Delta ^3 = \widetilde {M^3(G)}$$  is called {\bf{easy}} if there is a function $\mu : \mathbb Z^+ \to \mathbb Z^+ $ such that for each $g \Delta ^3 \subset \widetilde {M^3(G)}$ there are at most $\mu (\| g \|)$ fundamental domains $\overline g \Delta ^3$ such that $f(\overline g \Delta ^3) \cap g \Delta ^3 \neq \emptyset$. This makes the map $f$ \textsc{proper}. A given $g \Delta ^3$ is touched only by \textbf{finitely} many $\overline g \Delta ^3$'s.

\item[(II)] A 2-dimensional \textsc{representation} $f: X^2  \longrightarrow \widetilde {M^3(G)} $  is called {\bf{easy}} if  the subsets $M^2(f) \subset X \times X$ (and hence $M_2(f) \subset X^2$ too) and $f(X^2) \subset \widetilde {M^3(G)}$ are both \textbf{closed} subsets of the respective targets.
\end{itemize}
\end{defi}

  The reader should be careful, we use the word ``easy'', in this paper, both in the usual mundane sense and also in the technical sense of Definition \ref{easy} and of the next definition, the ``easy groups''. The two meanings are quite distinct and the context should always say which of the two meanings we actually use.

\begin{defi} \label{easy-group} A finitely presented group $G$ is called {\bf{easy}} if it admits easy \textsc{representations}.
 \end{defi}

  The interest of the concept easy is that any easy group $G$ is actually \textsc{qsf} \cite{Ot-Po1} (as we shall see in the sequel of this paper by the second author), and this also means that if by any chance $G = \pi _1 M^3$, then $\pi_1 ^{\infty} \widetilde M^3=0$ \cite{Ot-Ru-Ta2}. We will come back to these issues later.

  \begin{rem} A \textsc{representation} $f: X^p \to Y^3$ which is \textsc{not} easy exhibits
  the following phenomenon, which we call the \textbf{\textit{Whitehead nightmare}}, a terminology first introduced in \cite{Po-contMath}. And here is, precisely, what the Whitehead nightmare means, when it is present, for the \textsc{representation} above. Any compact $K \subset Y^3$ is hit \textbf{infinitely} often by the map $f$. This kind of thing is clearly there for the  Whitehead manifold $Wh^3$, hence the name.  But then, the Whitehead nightmare is also displayed by the \textit{Casson handles} and by the \textit{gropes} of \v{S}tan'ko, Edwards, Cannon, Freedmann and Quinn \cite{Free, Fr-Qu, Gu-Ma}.
  \end{rem}

  We end this section with the following useful fact (see \cite{Po_QSF1_Geom-Ded}):

  \begin{lem} The existence of an easy 3-dimensional \textsc{representation} implies the existence of an easy 2-dimensional \textsc{representation}.
  \end{lem}

\begin{rem}
  The converse is true too, but the proof is much more complicated \cite{Po_QSF3, PT4-AMS}. Anyway, in the context of Definition \ref{easy}, we always have (I) $ \Leftrightarrow$ (II).
\end{rem}

\section{Groups and REPRESENTATIONS}
Although general \textsc{representations} do always exist, easy ones are much harded to get. And one of the main tools for circumventing this difficulty is to put a geometric condition on the group, as originally done in \cite{Po3-JDG}. In what follows, we will provide some concrete example of classes of groups which are easy (in the sense of Definition \ref{easy-group}), and we will see the main features of 2- and 3-dimensional \textsc{representations} of general finitely presented groups.

  \subsection{Almost-convex groups}
  Let now $G$ be a finitely presented group, with a chosen finite system of generators $B = \{ g_1 ^{\pm 1}, g_2 ^{\pm 1}, \ldots, g_l ^{\pm 1}\}$.    It is not hard to show that we can always construct a \textsc{representation} (1 -- 15) (and/or (1--18--1), they are the same object), so that when $\Sigma = \{ \overline g_1 ,  \overline g_2 , \ldots,  \overline g_{2N}  \}$ is changed into a system of generators $\Sigma =  \{ \overline g_1 ^{\pm 1}, \overline g_2 ^{\pm 1}, \ldots, \overline g_N ^{\pm 1}\}$ of $G$, we have $B \subset \Sigma$. [One can adapt the proof of Lemma 3.2 from \cite{Po3-JDG} to the present more singular set-up].

  The system $B$ given, we have the \textit{Cayley graph} $\Gamma (G,B)$ with its word-length norm $\| \cdots \|$ induced by $B$, its geodesic arcs,  its balls of radius $n$, call them $B(n)$, and their boundaries $S(n) = \partial B(n)$.

  Here is now an interesting class of finitely generated groups. Following Jim Cannon \cite{Can}, the Cayley graph $\Gamma (G,B)$ is called \textbf{\textit{$k$-almost convex}} if there exists an $N(k) \in \mathbb Z^+$ such that for any $x,y \in S(n)$ which are such that $d(x,y)\leq k$, we can find a path of length $\leq N(k)$ in $B(n)$ joining them.

  If, for some $B$ and all $k$'s, the Cayley graph $\Gamma (G,B)$ is $k$-almost convex, then one says that $G$ itself is an \textbf{\textit{almost-convex}} group. Note that almost-convexity depends on the chosen set of generators \cite{Thi}. Also, it is easy to show  that almost
convex groups are not just finitely generated but finitely presentable \cite{Can}. All the hyperbolic groups of M. Gromov \cite{CDP, G-dlH, Grom-hyp} are almost-convex; so are also Euclidean groups, small cancellation groups, Coxeter groups, and discrete groups based on seven of the eight three-dimensional geometries; on the other hand, SOL
  groups  are not almost convex, and neither are solvable Baumslag-Solitar groups (see \cite{Can, CFGT, DaSh, HM,  MS, S-S}).

  \begin{rem}
  Concerning the notion of almost-convexity introduced above, notice that, without any condition, for any finitely presented  group $G$ and any of its Cayley graphs $\Gamma (G,B)$, and for any $x,y \in S(n)$, we can always find a path $\lambda \subset B(n)$ joining $x$ to $y$, with length$(\lambda) \leq 2n$. This is a triviality, of course.
  \end{rem}

  Now, there is also a so-called ``\textbf{Po Condition}'', intermediary between almost-convexity and the universal triviality above, namely:
  \begin{defi}\label{Po-condition}
  The finitely presented group $G$ satisfies the {\bf{Po condition}} if there exists a finite system of generators $B$ and two constants $\epsilon >0 , C>0$ such that for any pair $x,y \in S(n)$ with $d(x,y) \leq 3$, one can find a path $\lambda \subset B(n)$ joining $x$ with $y$ with $length \lambda \leq C n^{1-\epsilon}$.
  \end{defi}

  This condition also leads to interesting consequences, and see here \cite{Po3-JDG, PT2-k-weak} and the much more recent \cite{Kap}. Note also that  discrete  cocompact  solvgroups  do  not  fulfill  Po condition \cite{Fu_convex}.

 \begin{prop}\label{AC>easy}
 Let $G$ be a finitely presented group such that some $\Gamma (G,B)$ is 3-almost convex. Then $G$ is an easy group. Hence so are the hyperbolic and the NIL groups.
 \end{prop}

 \begin{proof}
 From papers like \cite{Po3-JDG, Po5-Top2, PT2-k-weak},  can be actually extracted a proof of the following
more general fact: if a finitely presented group satisfies a ``nice'' geometric condition (like
e.g. Gromov-hyperbolicity, almost-convexity, automaticity, combability etc.),
then it is easy.  \end{proof}

\begin{rem} Another proof of a slightly weaker form of Proposition \ref{AC>easy} can be deduced  from the recent papers \cite{Ot-Ru1, Ot-Ru-Ta2}. Although this proof is very short, it relies on other deep results. First of all in \cite{Ot-Ru1} it is proved that almost-convex groups are tame 1-combable (see the next section for definitions and more details) and hence  \textsc{qsf} by \cite{MT}. By the main result of \cite{Ot-Ru-Ta2}, one can infer that such a group admits an easy \textsc{wgsc-representation} with an additional finiteness condition. This is weaker than being an easy group (where it is required an easy \textsc{gsc-representation}). However,  in dimension 3, this condition actually implies the simple connectivity at infinity (see \cite{Ot-Ru-Ta2}), namely the original motivation of \cite{Po3-JDG}. \end{rem}

 \subsection{Combings of groups}

 The techniques of the section above can also be adapted and used for proving a similar result, but for a different class of discrete groups: namely {\textbf{\textit{combable groups}}} (see \cite{ECHDLPT} for precise definitions and details). Let $G$ be a finitely presented group and
let $S=S^{-1}\subset G$ be a finite set of generators for $G$. We will
denote by $\mathcal G \equiv \mathcal G (G,S)$ the \textit{Cayley graph}
of $G$ with respect to $S$ (notice also that the dependance on the generating set $S$
will be no longer stressed).

A useful geometric property one may impose to finitely generated groups is the existence of a
``nice'' combing, where, roughly speaking,  a combing assigns to each vertex of the
Cayley graph a set of paths from the identity to the vertex, so
that, as the vertex varies, the path varies in a `reasonable' way too,
i.e. nearby elements of $G$ give paths which are uniformly near
and whose domains are almost the same. To be more precise:
a {\textbf{\textit{combing}}} of $G$ is, by definition, a choice, for each $g \in G$, of a
continuous (not necessarily geodesic) path of the Cayley graph of $G$, joining $1$ to $g$. One can also think of this  path as a function $s_g: \mathbb Z ^+ \to G$  such that $s_g(0)=1, d(s_g(t), s_g (t+1)) \leq 1$, and, for all sufficiently large  $t$, we have $s_g(t) = g$.

Abstracting from the properties of the class of automatic groups, W. Thurston
called a combing {\textbf{\textit{Lipschitz}}} if there are two constants $C_1, C_2$ such that,
for all $g, h \in G$ and $t\in \mathbb Z^+$, we have $d(s_g(t), s_h(t)) \leq C_1 d(g,h) + C_2$.
 A finitely generated group admitting a Lipschitz combing is said to be {\textbf{\textit{combable}}}. Note that combable groups are actually finitely presented \cite{ECHDLPT}.

 By combining the terminology and the
 methods of the proof of the section above, together with Theorem 2 from \cite{Po3-JDG}, one can actually show that  combable groups (satisfying an additional technical condition) are easy.

 \begin{rem} Note that the class of combable groups does not  contain the full class of almost-convex groups, and  vice-versa (see Figure 1.1 of \cite{Po3-JDG}). For instance, the 3-dimensional Heisenberg group as well as $SL(n,\mathbb Z)$ are not combable \cite{ECHDLPT}.
 \end{rem}

 More recently, in \cite{Ot-Po2}, we have proved a similar result for a slightly different class of  groups admitting a nice combing. Before stating our result, we need to recall several definitions. Let $X$ be the universal cover
of the standard 2-complex associated to some finite presentation of $G$. Choose a base
point $e$ in its 0-skeleton $e \in X^0$.
The following more general definitions of combings are due to Mihalik and Tschantz \cite{MT}:

\begin{defi}\label{combings}
 A {\textbf{\textit{$0$-combing}}} of a 2-complex $X$ is a  homotopy
$\sigma\colon X^0 \times [0,1] \rightarrow X^1$ for which
$\sigma(x,1)=x$ for all $x \in X^0$, and
$\sigma(X^0,0)=e$, (where $X^1$ is the 2-skeleton of
$X$).

A {\textbf{\textit{$1$-combing}}}  of the 2-complex $X$ is a
continuous family of paths $\sigma_p(t)$, $t\in[0,1]$,
joining each point  $p$ of the 1-skeleton of $X$ to  $e$, whose restriction to vertices is a $0$-combing.
This is a homotopy
$\sigma\colon X^1 \times [0,1] \rightarrow X$ for
which $\sigma(x,1)=x$ for all
$x \in X^1$, $\sigma(X^1, 0)=e$, and
$\sigma |_{X^0 \times [0,1]}$ is a $0$-combing.
\end{defi}

In the same spirit  as above Mihalik and Tschantz replaced the Lipschitz  condition on combings by
the following property, that is of topological nature:

\begin{defi}\label{tame-combings}
A $0$-combing is called {\bf{tame}} if for every compact set
$C \subseteq X$ there exists a compact set
$K \subseteq X$ such that for each $x \in X^0$
the set $\sigma^{-1}(C) \cap (\{x\} \times [0,1])$ is contained in one path
component of $\sigma^{-1}(K) \cap (\{x\} \times [0,1])$.

\vspace{0.1cm}
\noindent A $1$-combing is {\bf{tame}}  if its restriction to the
set of vertices is a tame $0$-combing and for each compact
$C\subset X$ there exists a larger compact $K\subset X$ such that
for each edge $e$ of $X$, $\sigma^{-1}(C) \cap (e \times [0,1])$
is contained in one path component of $\sigma^{-1}(K) \cap (e \times [0,1])$.

\vspace{0.1cm}
\noindent A group is  {\bf{tame  1-combable}} if the universal cover of some
(equivalently any, see \cite{MT})
finite 2-complex with given fundamental group admits a tame 1-combing.
\end{defi}

 Groups whose Cayley complexes admit nice (e.g. bounded, tame or Lipschitz) combings have
good algorithmic properties, like automatic groups and hyperbolic
groups, and were the subject of extensive study in the last twenty
years (see \cite{ECHDLPT, HM} and \cite{MT}), and, for instance,
group combings were essential ingredients in Thurston's study  of
fundamental groups of negatively curved manifolds \cite{Thu}.

\begin{rem} Tame combings can be thought of as those combings which avoid the
 \textit{``Whitehead nightmare"} (see the next section), as defined in  \cite{Po-contMath}, and
often mentioned  since (that is, when some of the paths $\gamma _g$
of the combing  come back again inside the ball $B_n$ after a long time
(i.e. for $g$ very far)).
\end{rem}

 Here are some important  results and problems concerning tame combable groups:
 \begin{itemize}
 \item Almost-convex groups are tame 1-combable \cite{HM};

 \item Tame 1-combable groups are \textsc{qsf} \cite{MT};

 \item  All asynchronously automatic and semihyperbolic groups have tame 1-combings \cite{MT};

 \item If a finitely presented group $G$ has an asynchronously bounded, tame 0-combing, then $G$ has a tame 1-combing \cite{MT};

 \item Being tame 1-combable is a quasi-isometry invariant \cite{Br2};

 \item There are still no examples of
finitely presented groups which are not tame 1-combable;

\item It is conjectured that \textsc{qsf} groups are tame 1-combable
(and hence that all finitely presented groups are tame 1-combable).

\item If $ X$ is a finite complex and $\pi_1(X) =G$,  then $G$ has a tame 1-combing if and only if for each finite subcomplex $C$ of $X$, $\pi_1 (X-C)$ is finitely generated (in other terminology, the fundamental group at infinity of $G$ is pro-finitely generated) \cite{MT};

    \item If a group $G$ is not tame 1-combable, then $\pi_1 ^{\infty} (G)$ is not pro-finitely generated.

 \end{itemize}

 All this being said, we are finally able to state our main theorem from \cite{Ot-Po2}, whose proof is somehow similar to that of Proposition \ref{AC>easy}:

 \begin{prop}(\cite{Ot-Po2})\label{0-combings>easy}
A finitely presented group admitting a Lipschitz and tame
0-combing is easy.
\end{prop}

\begin{rem}\
\begin{itemize}

\item If the group $G$ has one end, then there always exists a tame
0-combing for the Cayley graph of $G$.
This means that the additional Lipschitz condition in  Proposition
\ref{0-combings>easy} is necessary.

\item Proposition \ref{0-combings>easy} does not follow from  the results of
\cite{Po3-JDG}. There, no tameness hypothesis is considered, while, on the other hand,
 there is the additional condition that any two points of the boundary of
$n$-balls can be joined by paths of length $\leq C n
^{1-\epsilon}$, for some constant $C$.

\item Our result may be compared with  Lemma 4.1 of \cite{MT}, which states that if a finitely presented group $G$ has an asynchronously bounded, tame 0-combing, then $G$ has a tame 1-combing and hence is \textsc{qsf}.
\end{itemize}
\end{rem}

 Now, an interesting thing to do in the next future would be to try to generalize
Proposition \ref{0-combings>easy} for (the far more general) tame 1-combable groups, with no Lipschitz or boundedness
condition. Our feeling is that a tame 1-combing is very close to a
Lipschitz tame 0-combing, and hence, we do believe the techniques
we used in \cite{Ot-Po2} have a good chance to work also for such groups.

\subsection{The Whitehead nightmare}

It is now a good moment to say something about the Whitehead nightmare which we have already mentioned
several times, and about  its connection with finitely presented groups. Consider a 3-dimensional \textsc{representation}  $f: X^3 = \bigcup _{i\in I} \Delta _i \longrightarrow \widetilde {M^3(G)}=
{\bigcup}_{\gamma \in G} \ \gamma \delta$. Generally speaking, without any other additional information and/or conditions, one normally finds the
following situation, which, in papers like \cite{Po-contMath}, the second
author (V.P.) has called the {\textbf{\textit{Whitehead nightmare}}}:
$$ \# \big \{ \Delta _i , \mbox { s.t. } f(\Delta _i) \cap
\gamma \delta \neq \emptyset \big\} = \infty, \ \forall \gamma \in
G .$$
So, the Whitehead nightmare, for a \textsc{representation}, means that $f(X^3)$ explores infinitely many times any nook and hook of $\widetilde {M^3(G)}$. We feel that this is somehow reminiscent of what the Feynman path integral does.
The 2-dimensional analogue of the condition above is the
following condition:
$$ M_2(f) \subset X^2  \mbox { is \textbf{not} closed}.$$
This is the common  situation for standard 2-dimensional
\textsc{representations} and one has to start by living with
it and look at the accumulation pattern of $M_2(f)$ inside $X^2$.

However, in \cite{Ot-Po3}, we were able to prove the following result, that practically tells us that all finitely presented groups admit sufficiently nice \textsc{representations} avoiding a milder version of the Whitehead nightmare:
\begin{thm}[D.E. Otera - V. Po\'enaru, \cite{Ot-Po3}]\label{WN}\
\begin{enumerate}

\item (3$^d$-part)  For any finitely presented \textsc{qsf} group G, there exists a
locally finite 3-dimensional  \textsc{wgsc-representation}, $X^3 \overset
{f} {\longrightarrow} \widetilde {M^3(G)}$, such that the following conditions are satisfied for any $\gamma \in
G$:

There is a free action $G \times X^3 \longrightarrow X^3$, and  $f$  is \textbf{equivariant}, and
there is a constant $ C = C(\|\gamma\|) > 0$  (depending on the word-length $\| \cdots \|$ of $\gamma$)  such that
 one has $$\# \{\Delta _i \
| \  f(\Delta_i) \cap \gamma \delta  \neq \emptyset \} < C. $$

In particular, any given domain $\gamma \delta \subset \widetilde {M^3(G)} =
\bigcup _{\gamma \in G}  \gamma \delta$
downstairs, can only be hit \textbf{finitely many} times by the image of a large domain
$\Delta \subset X^3 = \bigcup _{j\in J} \Delta _j$ from upstairs.

\item (2$^d$-part) For any finitely presented \textsc{qsf} group $G$, there exists a
locally finite 2-dimensional \textsc{wgsc-representation} $X^2 \overset
{f} {\longrightarrow} \widetilde{M^3(G)}$ which is both
\textbf{equivariant}, and which also satisfies the following condition:
$$ \mbox{ Both } f(X^2) \subset \widetilde {M^3(G)}
 \mbox { and } M_2(f) =\big \{ x \in X^2 \ | \ \sharp \{f^{-1} (f(x))\} > 1\big \} \subset X^2 \mbox { are \textbf{closed}}.
 $$
\end{enumerate}
\end{thm}

\begin{rem}\
\begin{itemize}

\item The finiteness condition of point 1. above can also be replaced by the following
variant: there exist \textbf{equivariant triangulations} for
$\widetilde{M^3(G)}$ and for $X^3$, and also a constant $C'$
such that, for any simplex $\sigma \subset \widetilde {M(\Gamma)}$,
we should have
$$ \# \big \{\mbox {simplexes } S \subset X^3 \mbox {, s.t. }
f(S) \cap \sigma \neq \emptyset  \big\} < C'. $$

\item We conjecture that the same result can be improved for \textsc{gsc-representations}. This would imply that being easy is equivalent to the \textsc{qsf} property for finitely presented groups (for a partial result see \cite{Ot-Ru-Ta2}).

    \item To sum up, one should understand that the presence of a discrete group action somehow should force \textsc{representations}
to be easy, namely to avoid the Whitehead nightmare.

\end{itemize}
\end{rem}

 \subsection{Equivariant REPRESENTATIONS and the uniform zipping length}
 Let us move now to more serious, adult stuff.

 \begin{thm}[Existence of \textbf{equivariant} \textsc{representations}; V. Po\'enaru, \cite{Po_QSF1_Geom-Ded}]\label{uniform}
  Let $G$ be any finitely presented group. There is then a 3-dimensional \textsc{representation}
  $$ f : X^3 \longrightarrow \widetilde {M^3(G)}$$
  such that:
  \begin{enumerate}
  \item The singular handlebody $X^3$ is \textbf{locally finite}.
\item There is a free action $$G \times X^3 \to X^3$$ such that $f$ is \textbf{equivariant}, meaning that for any $x \in X^3$, $g \in G$, we have  $$f (gx)=  g (f(x).$$
\end{enumerate}
 \end{thm}

 Contrary to the elementary, Kindergarten, Lemma \ref{lemma7}, this theorem requires considerably more work. Complete proofs are to be found in \cite{Po_QSF1_Geom-Ded} and see also \cite{Po_QSF_survey} and \cite{PT4-AMS}. In the next section we will provide some general ideas for the proof.

\vspace{3mm}

 This is a good place for explaining why, in places like in Definition \ref{repr}, we do not use the mundane 2-dimensional presentations of groups $G$, but rather our 3-dimensional presentations $M^3(G)$. Here is at least one good reason for this.

 For any kind of geometric presentation of a group, we certainly need cells, but call them now rather handles, of index $\lambda =0, \lambda =1, \lambda =2$ and of some dimension $n$. But when one wants to prove something like our Theorem \ref{uniform}, equivariance actually gets \textsc{forced}, as we shall soon explain (and see \cite{Po_QSF1_Geom-Ded, PT4-AMS} for a more detailed meaning of this ``forcing''). In the forcing process, to be explained a bit more in the additional comment to this section,  we meet the following problem: if one does things naively, then \textbf{infinitely} many handle attachments accumulate, or pile up, on the \textbf{same spot}. This would create a kind of non-local finiteness, which we could not live with. As we shall see, the cure for this disease asks us to start by sending the lateral surfaces $\delta H$ of the handles $H$ which we use, to infinity, and this eventually leads to a nicely convergent process, once the handle attachments are performed at higher and higher levels, closer and closer to the infinity where the lateral surfaces (now missing) have been sent. There is here a whole technology of so-called ``bicollared handles'' \cite{PT4-AMS}, to which we will come back in the second part of this survey by the second author.

 But then, handling the $\delta H$'s the way we just did, needs co-cores of positive dimensions (since $\partial ($cocore$) \approx  \delta H$), and this forces $n\geq 3$ on us. Next, when \textsc{representations} have to be used, in particular in the second part of the trilogy \cite{Po_QSF1_Geom-Ded, Po_QSF2, Po_QSF3}, we will have to investigate any nook and hook of  dense subspaces in the presentation space, with our zipping. And it is exactly zippings of 2-complexes into 3-dimensional $\widetilde {M^3(G)}$'s which come with a very rich structure of which we want to take advantage of. And this excludes the $n>3$'s too. So we are finally forced with $n=3$, and this not because of any particular love for 3-manifolds. Our $\widetilde {M^3(G)}$'s are anyway not 3-manifolds, but singular spaces.

 \begin{figure}

\centering
\includegraphics[width=110mm,scale=0.7]{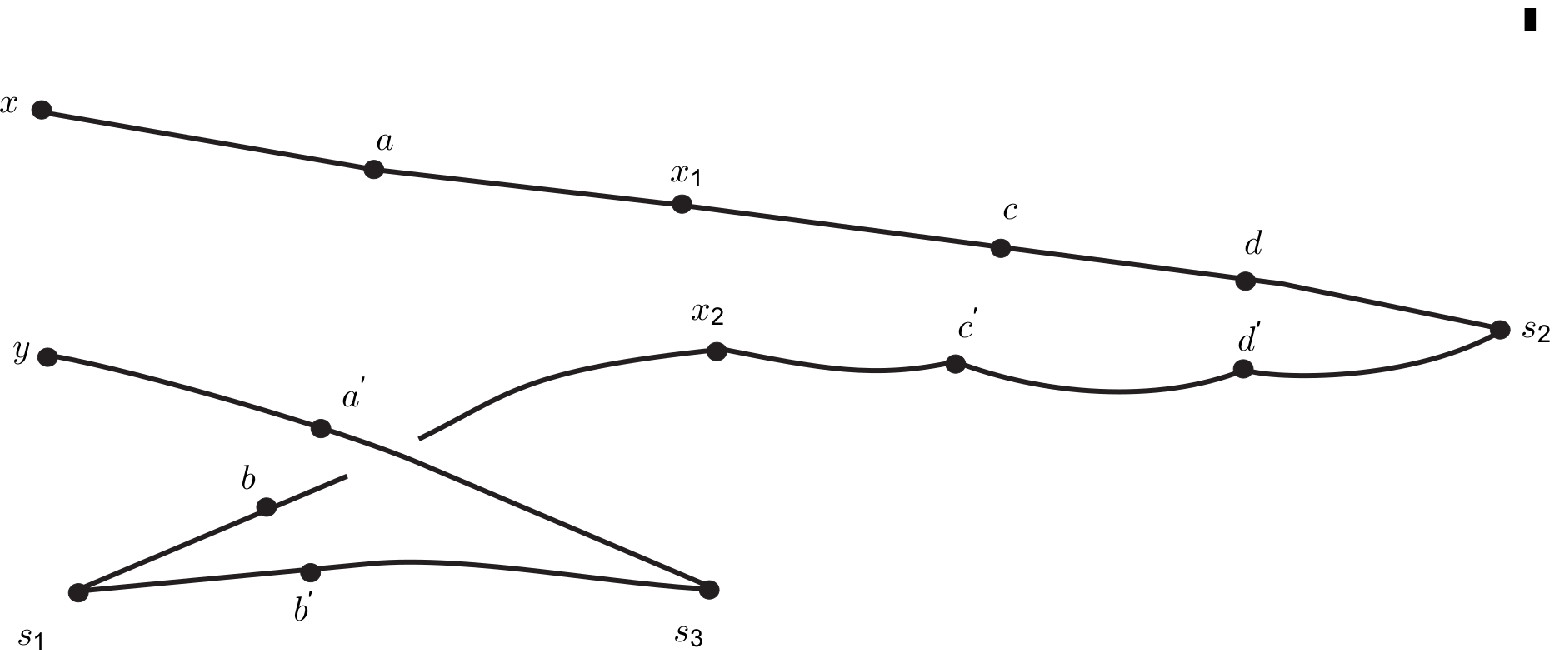}

\caption{A zipping path $\lambda (x,y) \subset \hat{M}^2 (f)$. Here $s_1, s_2, s_3 \in \rm{Sing} (f)$ and $(x_1, x_2, s_3)$ stands for a triple point on $M^3(f)$ and
$(x,y) \in M^2(f)$. Similarly, all the pairs like $(a, a')$, $(b,b') \in \lambda (x,y)$ are in $M^2(f)$. } \label{fig:6}

\end{figure}

 Before we can describe the next more serious result, let us go back to the condition $\Phi (f) = \Psi(f)$ which occurs in Definition \ref{repr}, and right now we will continue to stick to 3-dimensional \textsc{representations} of $G$. For $f: X^3 \to \widetilde {M^3(G)}$ we have $\rm{Sing} (f) \subset X^3$, the mortal singularities, not to be mixed up with the immortal singularities of $\widetilde {M^3(G)}$. With this, let
  $$ \Phi (f) \supset M^2(f) \subset X^3 \times X^3 \supset \rm{Sing} (f), $$ where the simplified notation $\rm{Sing} (f)$ means $\rm{Diag}(\rm{Sing} (f))$. Let also denote
  $$\hat M^2(f)= M^2(f) \cup \rm{Sing} (f).$$
   With all these things, the condition $\Psi(f) =\Phi(f)$ can also be expressed as follows: any $(x,y) \in M^2(f)$ can be joined, inside $\hat M^2(f)$, by a continuous path to the singularities. We call such a path a {\textbf{\textit{zipping path}}} $\lambda (x,y) \subset \hat M^2(f)$, or, also, a zipping strategy. Without trying to be more pedantic concerning this definition, Figure \ref{fig:6} should give a clear picture of what a zipping path means.

   So retain that the condition $\Phi(f)= \Psi(f)$  means that for every $(x,y) \in M^2(f)$ there is a zipping path $\lambda (x,y)$. This is certainly not unique for a given $(x,y)$.

Our $M^3(G)$ is a metric space. On the various handles one can define Riemannian metrics compatible on intersections. This lifts then to metrics on $\widetilde {M^3(G)}, X^3, X^3 \times X^3$, which are well-defined up to quasi-isometry.  With this, we can put up the following

\begin{defi}\label{zipping-lenght}
For $(x,y) \in M^2(f)$, $$l(x,y) = \underset{\lambda}{\rm{inf}} \{ \rm{length\  of\  the\  zipping\  path\ } \lambda (x,y) \}.$$

\end{defi}

\begin{thm}[Uniform zipping length; V. Po\'enaru, \cite{Po_QSF1_Geom-Ded}]\label{uniform-zipping-lenght}
For any finitely presented group $G$, there is a \textsc{representation} $f: X^3 \to \widetilde {M^3(G)}$ which is not only locally compact and equivariant, like in Theorem \ref{uniform}, but also has a uniformly bounded zipping path, i.e. there is a $K>0$ such that for all $(x,y) \in M^2(f)$ we have $l(x,y) <K$.
\end{thm}

The proof is to be found in \cite{Po_QSF1_Geom-Ded}. Just like for equivariance, the bounded zipping length has to be \textbf{forced}. The group structure is essential for the proof and we conjecture that for $Wh^3$ there is \textsc{no} bounded zipping length. Also, strangely enough, the proof of Theorem \ref{uniform-zipping-lenght} requires an infinite process.

\subsection{Comments concerning the proofs of Theorems \ref{uniform} and \ref{uniform-zipping-lenght}}

Concerning the \textbf{locally finiteness} in Theorem \ref{uniform}, in order to achieve it, we have to use the so-called ``bi-collared handles'' which we will describe a bit more in Part II of this survey.

Topologically speaking, a \textit{bi-collared handle} is a 3-dimensional handle with the lateral surface removed (or, rather, sent to infinity). There are also two layered structures in this 3-dimensional object, one coming from the attaching zone, the other one going towards that removed lateral surface. With this, when a $(\lambda +1)$-handle is to be attached to the union of all handles of index $\mu \leq \lambda $, then the various intersections
$\{$attaching zones of $(\lambda +1)$-handles$\} \cap \{$a given $\mu$-handle$\}$ (and for a given $\mu$-handle there are infinitely many of them), can be corralled towards infinity, insuring thus our local finiteness.

We move now to the equivariant part of Theorem \ref{uniform}, and there we need a \textbf{symmetry-forcing} procedure of which we only present here an \textsc{abstract} version. In \cite{Po_QSF1_Geom-Ded} this piece of abstract nonsense is transformed into geometry.

So, we are given a simply-connected simplicial complex $X$, endowed with a free action of some discrete group $G$, $G \times X \to X$. We
are also given a second simplicial complex $Y$, coming with a non-degenerate simplicial map $F: Y \to X$. Think here of $G \times X \to X$ as being our $G \times \widetilde {M^3(G)}\to \widetilde {M^3(G)}$ and of $F: Y \to X$ as being some locally-finite \textsc{representation}, not yet equivariant.

If $X^{(\epsilon)}$, $Y^{(\epsilon)}$ are the respective $\epsilon$-skeleta, it is assumed here that the map $F: Y^{(0)} \to X^{(0)}$ is surjective, with countable fibers. For the simplicity of the exposition, we assume $X$ and $Y$ to be 2-dimensional.

\begin{prop}
There is an extension $F_1: Z \to X$ of $F: Y\to X$, coming with the commutative diagram
$$\begin{diagram}
Y&&\rInto&&Z \\
    &\rdTo_{F}& &\ldTo_{F_1}\\& &X
\end{diagram} $$

\noindent which is such that:
\begin{itemize}

\item The map $F_1$ is simplicial and non-degenerate;

\item $Z$ is geometrically simply connected;

\item There is a free action $G \times Z \to Z$ such that $F_1$ is equivariant.

\end{itemize}
\end{prop}

\begin{proof}[Sketch of proof] We choose \textbf{arbitrarily} an indexing by the natural integers $\mathbb Z_+$
of each fiber $F^{-1} (x), \ x\in X^{(0)}$, i.e. we trivialize $Y^{(0)}$ as $Y^{(0)}= X^{(0)} \times \mathbb Z_+$.
 We \textsc{force} then an action $G \times Y^{(0)} \to Y^{(0)}$, by $g(x, z_+) = (gx, z_+)$.

 Next, we construct an intermediary simplicial space $Y \subset Z_0 \subset Z$ with the following spare parts
 $Z_0^{(0)} = Y^{(0)}$, $Z_0^{(\epsilon)} = G \times Y^{(\epsilon)}$, for $\epsilon = 1, 2$.

 We put together these spare parts as follows. Whenever at level $Y$ we have the face-operation
   $ \nu_i ^{(\epsilon)}:Y^{(\epsilon)} \to Y^{(\epsilon -1)}$, then we \textsc{force} the face relations:
    $ \nu_i ^{\epsilon}(g s^{\epsilon} )= g (\nu_i ^{\epsilon} s^{\epsilon})$, for $Z_0$.

    Concretely, this means that, when at level $Y$ we have things like
    $$ \partial s ^1 = s_1^0 - s_2^1, \ \mbox{ or } \partial s^2 = s_0^1 - s_1 ^1 + s_2^1, $$
    then at level $Z_0$ we have
$$\partial gs ^1 = gs_1^0 - gs_2^1, \mbox{ and } \partial gs^2 =g s_0^1 - gs_1 ^1 + gs_2^1,\mbox{for any } g\in G.$$
Moreover, we can define the map $F_0: Z_0 \to Y$ by $F_0(gs^{\epsilon}) = g F_0 (s^{\epsilon})$. This map is certainly equivariant but $Z_0$ is still not \textsc{gsc}.

So here is how we force this too, continuing to stay equivariant (and the concrete geometric realization
of this abstract nonsense story will stay locally-finite too).

Let $T\subset Z_0^{(1)}$ be a maximal tree, which misses, exactly, the edges $e_1, e_2, \ldots $ of $Z_0$. In $Z_0^{(1)}$, we chose then a family of closed loops $\lambda _i$ with geometric intersection matrix $\lambda _i \cdot e_j= \delta _{ij}$.

Let $D^2_i$ be an abstract disk cobounding $\lambda _i$. In $Z_0$ we have loops $g \lambda _i$ for every $i$ and every $g\in G$. With this we construct the space $$Z \equiv Z_0 \underset{\sum g \lambda _i }{\bigcup} g D^2_i.$$

Clearly $Z$ is \textsc{gsc} and, with $\pi_1 X=0$, the equivariant map $F_1: Z \to X$ should now be obvious. \end{proof}

All this was purely abstract stuff, and to prove Theorem \ref{uniform} we need to combine it with the technology of bi-collared handles, making the local-finiteness also achieved. But
there is here a bit more to be done.

If we leave things as they are now, we have nasty Cantorian accumulation transversal to the double line, like in the next section. We avoid it by an appropriate ``decantorianization'', which we will not describe here. And we need to avoid them for the proof that all
groups are \textsc{qsf}.

With this, we have finished with Theorem \ref{uniform} and we can comment now Theorem \ref{uniform-zipping-lenght}. Actually, both Theorem \ref{uniform} and Theorem \ref{uniform-zipping-lenght} are important steps in the proof that all groups are  \textsc{qsf}, see here Part II of this survey.

Let us say that, when Theorem \ref{uniform} has been proved, it leaves we with a \textsc{representation}
with all the features from the Theorem \ref{uniform} in question, namely the following:
$$f: X^3 \to \widetilde {M^3(G)}. \e{(1-29)}$$
There is no reason, of course, that for the double points
$$ (x,y) \in M^2(f) \subset X^3 \times X^3, \e{(1-30)}$$
and for their zipping paths, we should have uniformly bounded zipping length. So here is what we do
about that.

Staying locally-finite and equivariant, we add now bi-collared handles to (1 -- 29), moving to a new
\textsc{representation} which extends (1 -- 29):
$$f_1: X_1^3 \to \widetilde {M^3(G)}, \e{(1-31)}$$
where, by creating additional zipping paths, we \textbf{force} the double points of (1 -- 29),
$M^2(f) \subset M^2(f_1) \subset X_1^3 \times X_1^3$, to have uniformly bounded zipping length.
 The uniform bound is here a simple function of the size of the fundamental domain of the basic
 action $G \times  \widetilde {M^3(G)} \to  \widetilde {M^3(G)}$.

 Next, by a second extension, we force the double points in $M^2(f_1) - M^2(f)$ to have uniformly
 bounded zipping length too. But then, also, more  uncontrolled  double points have been created.

 So, we need a  full  infinite sequence of bigger and bigger \textsc{representations}
 $f_n : X^3_n \to  \widetilde {M^3(G)}, n=1,2,3 \ldots$ coming with the commutative diagram
 $$\begin{diagram}
X^3_n & \rTo^{f_n} &&\widetilde {M^3 (G)} \\
    \dInto & & \ruTo_{f_{n+1}}(3,2) &\\
    X^3_{n+1}& & &\end{diagram}$$
 and such that for any double point $(x,y) \in M^2(f_n) \subset M^2(f_{n+1})$ there is a zipping path
 of uniformly bounded zipping length in $f_{n+1} : X^3_{n+1} \to  \widetilde {M^3(G)}$,
 leaving the rest of the double points $ M^2(f_{n+1}) - M^2(f_{n})$ completely uncontrolled. That uniform bound is,
 of course, $n$-dependant. The big issue is to make this infinite process converge, see \cite{Po_QSF1_Geom-Ded}
 for details.

 And retain also that an \textsc{infinite process} is actually necessary for achieving the desired
  uniformly bounded zipping length, staying locally-finite and equivariant, at the same time.
  (Of course some work is necessary for achieving all these things, simultaneously, see \cite{Po_QSF1_Geom-Ded}).

 \section{Appendix: a chaotic 2-dimensional REPRESENTATION of the Whitehead manifold}

 The present section is not concerned with groups nor with group actions, but it should show how wild the accumulation pattern for the double points set of a \textsc{representation} can be.
 And, in real life, i.e. in geometric group theory,  phenomena  like those exhibited now, will be
 avoided, like hell, via the so-called ``decantorianization''. [An idea of the process in question is suggested by Figure 2.7.1 in \cite{PT4-AMS}, a figure which is directly related to the material that we will present now].

 \begin{thm}[V. Po\'enaru - C. Tanasi, \cite{PT-Wh}]\label{repr-Wh^3}
 There exists a locally-finite, 2-dimensional \textsc{representation} for the
  Whitehead manifold $Wh^3$, call it
  $$ f: X^2 \to Wh^3, \e{(1-32)}$$
  with the following features:
  \begin{itemize}
  \item The $X^2$ is gotten from a point by infinitely many dilatations. (We say it is ``collapsible'',
  a property stronger than \textsc{gsc}. So $X^2$ is certainly \textsc{gsc} too).

  \item We have $M_3(f) = \emptyset$, and each connected component of the $M_2(f)$ is a finite tree, based at a mortal singularity.

  \item There are contact lines $\Lambda \subset X^2$ which are \textbf{tight transversals} (a
  term to be soon explained), to $M_2(f) \subset X^2$, such that $\Lambda \cap M_2(f)$ accumulates
  on a Cantor set.
  \end{itemize}
 \end{thm}

 \begin{figure}

\centering
\includegraphics[scale=0.9]{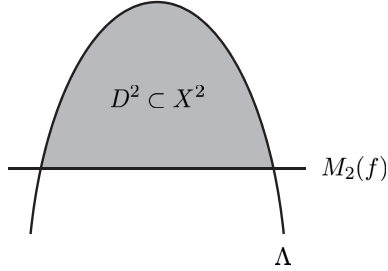}

\caption{Non-tight transversal.} \label{fig:7}

\end{figure}

 We shall sketch the proof of this theorem below, but the full proof and explanations concerning
 its connections to chaos theory are presented in \cite{PT-Wh}. Here are, to begin with, some comments:

 \begin{rem}\
 \begin{itemize}
 \item The fact that the transversal $\Lambda$ is tight means that Figure \ref{fig:7} is forbidden for it.
 \item  The accumulation pattern of $(\Lambda \cap M^2(f))$ is governed by a dynamical pattern which is the same as  that of a Julia set for a real quadratic polynomial map $f_{\lambda}$ with $\lambda$ outside of the Mandelbrot
     set (the set of those $\lambda$'s for which $\lambda \in \mbox{ Julia }(f_{\lambda})$), i.e. letting things become complex. And, as it is well-known, for  $\lambda$ outside of the Mandelbrot
     set, Julia$(f_{\lambda})=$ Cantor set.

     We will sketch below this dynamical pattern, without going into its connection with Julia and
     Mandelbrot, i.e. its connection with chaos. We will just record here that, years ago, it was John Hamal Hubbard who opened the eyes to the
     second author (V.P.) on the connection between his work on \textsc{representations} and the
     Julia sets.

     Finally, we believe that there is here a big connection between wild low-dimensional topology
     and chaos theory, a connection undoubtedly deserving to be explored.

     \item In a certain sense, the \textsc{representation} (1 -- 32) above is the simplest
     2-dimensional \textsc{representation} of $Wh^3$. Of course, the Whitehead manifold is very far from
      the highly symmetrical objects $\widetilde {M^3(G)}$'s with which the present survey paper deals,
      and, the kind of chaotic behaviour which (1 -- 32) has, is carefully avoided for the
      \textsc{representations} of  $\widetilde {M^3(G)}$ (see here Theorem 5 in Part II).
 \end{itemize}
 \end{rem}

 We go now to our Theorem \ref{repr-Wh^3} above. For pedagogical-expository reasons, we will not exhibit here the collapsible (1 -- 32), but a variant of it
 $$ f_1: K^2 \to Wh^3, \e{(1-33)}$$
 with a $K^2$ which is not collapsible, but only contractible. All the other features of (1 -- 32) will
 be there, nevertheless. It is actually not hard, afterwards, to go from (1 -- 33) to the collapsible
  (1 -- 32), by some  unglueing  of double lines $f_1(M^2(f_1))$; basically, one has to change
   the Figure 2.6 in \cite{PT-Wh} to Figure 2.8 in the same paper.

   In order to explain the (1 -- 33), we start with a modification of the construction from Figure \ref{fig:2} and,
   for convenience, we will change the notations $D_1, C_1, b,c$ from Figure \ref{fig:2} into
   $$ D_n, C_n, b_n, c_n . \e{(1-34)}$$
   Here comes now a more serious change. At the level of our Figure \ref{fig:2}, the disk $D_1$ was glued to
   itself along $[b, c]$. Now, instead of letting $D_n$ be glued to itself along $[b_n, c_n]$,
   we only let it stay glued along the shorter $[a'_n, a''_n]$ from Figure 2.5 in \cite{PT-Wh}. With
   this, the fat points $a'_n, a''_n$ become mortal (or undrawable) singularities and the dotted lines in Figure 2.5
   of \cite{PT-Wh} are in $M_2(f_1)$. Also, at the level of $D_n$, which from now on is conceived as being self-glued along $[a'_n, a''_n]$, the line $C_n$ is a simple closed loop. With these things our
   $K^2$ is the infinite union of the disks $D_n$, namely $K^2= D_1\cup D_2 \cup D_3 \cup \ldots$
   where the disks $D_i$ are like in Figure 2.4 of \cite{PT-Wh}.

   And, with this, comes now the basic identification, or gluing process of $D_{n+1}$ to $D_n$:
   $$ C_n \equiv \partial D_{n+1}, \forall n . \e{(1-35)}$$
   Figure 2.9 in \cite{PT-Wh} should suggest how the double point set $M_2(f_1)$  starts
    occurring, when $D_1$ and $D_2$ interact with $D_0$ and then $D_2$ with $D_1$. Of course, this is only the beginning of an infinite cascade where,
    $$ \mbox { for each } n \mbox { the } D_n \mbox { interacts with } D_{n+1}, D_{n+2}, \ldots . \e{(1-36)}$$
    When one looks at the infinite cascade unleashed in Figure 2.9 from \cite{PT-Wh}, we see double points of $f_1$ occurring on each $C_n$ and each $\partial D_n$, typically represented
    as a letter, $p$ or $q$, with upper and lower indices. Keep in mind that, because of (1 -- 35), the same double point occurs both  on  $C_n$ and on $\partial D_{n+1}$.

    For a given double point, we will use the definitions
    $$ \{ \mbox{flavour of the double point}\} \equiv \{\mbox{the number of lower indices} \} + 1 , \mbox { and } \e{(1-37)}$$
    $$\{ \mbox{color of the double point}\} \equiv \{\mbox{the letter }  (p \mbox{ or } q ), \mbox { and the unique upper index}\}. $$

    The canonical isomorphism $\partial D_{n+1} = C_n$ is both color and flavour preserving.

    The schematical Figure \ref{fig:8} can help in visualizing what is going on. In this figure,
    $s_n$ stands for the singularities $a_n'$ or $a_n''$. Also, the numbers which occur
    on the lower line are connected with the index of the singularity $s_n$ which generates
    them, and have nothing to do with color or flavour.

    \begin{figure}

\centering

\includegraphics{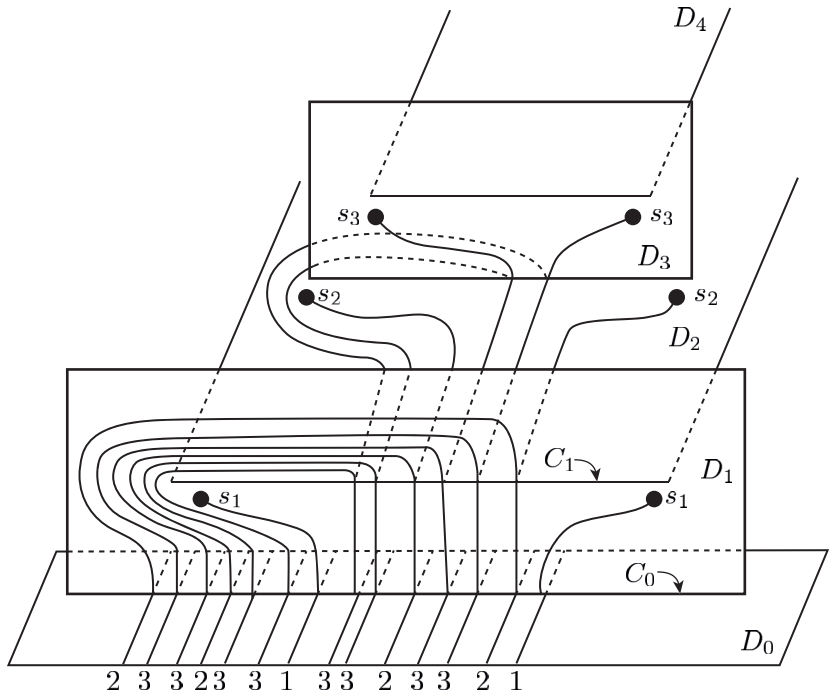}

\caption{The dynamical algorithm.} \label{fig:8}

\end{figure}

    We will end this discussion with the explicit description of the dynamical algorithm
    which generates the set $M_2(f_1) \cap \Lambda$, for a typical tight transversal
    $\Lambda = C_n$.

    We start with Figure 2.10 in \cite{PT-Wh} where, to begin with, we
    have just the generic fat lines $C, \partial D$. We will explain how the rest of the
    figure is generated by the Julia dynamical algorithm for generating $M_2(f_1) \cap \Lambda$,
     and, eventually, the numbers which will start popping up both on $C$ and on $\partial D$,
     will correspond to the flavours in (1 -- 37).

    We have the initial points $a_{up}$ and $a_{down} \in C$. Then, starting from them,
     we draw horizontal dotted lines to $\partial D$, generating the two points $1 \in
     \partial D$. By dynamical self-consistency, we also draw the two $1\in C$, in agreement
     to the isomorphism $\partial D = C$, which becomes now flavour preserving.
     From the two points $1 \in C$ we generate next four points $2 \in \partial D$ and
     then this process continues indefinitely.

     Much more details concerning this story are provided in \cite{PT-Wh}.


\begin{thebibliography}{[ASMR999]}



\bibitem{Ag}
I. Agol, {\em  The virtual Haken conjecture}, Doc. Math., J. DMV {\bf{18}} (2013),  1045--1087.





\bibitem{BBBMP}
L. Bessi\`eres, G. Besson, M. Boileau, S. Maillot and J. Porti.
Geometrisation of 3-manifolds. EMS Tracts in Mathematics, volume
13. European Mathematical Society, Zurich, 2010.






\bibitem{Br1}
S.G. Brick, {\em {Quasi-isometries and ends of groups}}, J. Pure
Appl. Alg. {\bf{86}} (1993), No. 1, 23--33.


\bibitem{Br2}
S.G. Brick, {\em Quasi-isometries and amalgamations of tame
combable groups},  Inter. J. Alg. Comp.  {\bf{5}} (1995), No. 2,
199--204.



\bibitem{BM1}
S.G. Brick and M.L. Mihalik, {\em The QSF property for groups and
spaces},  Math. Zeit.   \textbf{220} (1995), No. 2, 207--217.


\bibitem{Can}
J.W. Cannon, {\em Almost convex groups},
Geom. Ded. {\bf{22}} (1987), 197--210.

\bibitem{CFGT}
J.W. Cannon, W.J. Floyd, M.A.  Grayson and W.P Thurston, {\em
Solvgroups are not almost convex},
Geom. Ded. {\bf { 31}} (1989), No. 3, 291--300.

\bibitem{CDP}
M. Coornaert, T. Delzant and A. Papadopoulos.
G\'eom\'etrie et th\'eorie des groupes. Les groupes hyperboliques de Gromov. Lecture Notes in Mathematics, 1441. Berlin etc.: Springer-Verlag. x, 165 p., 1990.


\bibitem{Da}
M.W. Davis, {\em Groups generated by reflections and aspherical
manifolds not covered by Euclidian Spaces}, Ann. Math. (2)
\textbf{117} (1983), 293--324.



\bibitem{DaSh}
M.W. Davis and M. Shapiro {\em
Coxeter groups are almost convex},
Geom. Ded. {\bf {39}} (1991), No. 1, 55--57.


\bibitem{Ed}
C.H. Edwards, {\em Open 3-manifolds which are simply connected at
infinity},  Proc. Am. Math. Soc.  {\bf{14}} (1963), 391--395.



\bibitem{ECHDLPT}
D.~Epstein, J.W. Cannon, D.~Holt, S.~Levy, M.~Paterson, and W.P. Thurston.
Word processing in groups.
Boston, MA etc.: Jones and Bartlett Publishers, 1992.



\bibitem{Free}
M.H. Freedman,  {\em The topology of four-dimensional manifolds},
J. Diff. Geom. {\bf{17}} (1982), 357--453.





\bibitem{Fr-Qu}
M.H. Freedman and F. Quinn. Topology of 4-manifolds.
Princeton Mathematical Series, vol. 39, Princeton University
Press, Princeton, NJ, 1990.


\bibitem{Fu_convex}
L. Funar, {\em  On discrete solvgroups and Po\'enaru condition}, Arch. Math. {\bf {72}} (1999),  81--85.



\bibitem{FG}
L. Funar and S. Gadgil, {\em  On the geometric simple connectivity
of open manifolds}, I.M.R.N. \textbf{24} (2004), No. 24, 1193--1248.





\bibitem{FO2}
L. Funar and D.E. Otera, {\em On the \textsc{wgsc} and
\textsc{qsf} tameness conditions for finitely presented groups},
Groups Geom.  Dyn. \textbf{4} (2010), No. 3, 549--596.





\bibitem{Gab}
D. Gabai,  {\em Valentin Po\'enaru's program for the Poincar\'e conjecture}, in
Yau, S.-T. (ed.), ``Geometry, topology and physics for Raoul Bott''. Cambridge, MA: International
Press. Conf. Proc. Lect. Notes Geom. Topol. 4 (1995), 139--166.


\bibitem{Gab_2}
D. Gabai,  {\em The Whitehead manifold is a union of two Euclidean spaces},
J. Topol. {\bf {4}} (2011), No. 3, 529---534.

\bibitem{Ger-St}
S.M. Gersten and J.R. Stallings, {\em Casson's idea about
$3$-manifolds whose universal cover is $\mathbb{R}^3$}, Inter. J.
Alg. Comp. {\textbf{1}} (1991), No. 4, 395--406.




\bibitem{G-dlH}
E. Ghys and P. de la Harpe (Eds.), {\em Sur les
groupes hyperboliques d'apr\`es M. Gromov}, Progress in Math.,
vol. {\bf{3}}, Birkhauser, 1990.

\bibitem{Glimm}

J. Glimm, {\em Two Cartesian products which are Euclidean spaces}, Bull. Soc. Math. France {\bf{88}} (1960), 131--135.


\bibitem{Grom-ICM}
M. Gromov, {\em
Infinite groups as geometric objects}.
Proc. Int. Congr. Math., Warszawa 1983, Vol. 1 (1984), 385--392.


\bibitem{Grom-hyp}
M. Gromov, {\em  Hyperbolic groups}, Essays in Group Theory (S.
Gersten Ed.), MSRI publications, n. {\bf{8}}, Springer-Verlag,
1987.




\bibitem{Grom-AsInv}
M. Gromov, {\em  Asymptotic invariants of infinite groups},
Geometric group theory, Vol. 2 (Sussex, 1991), 1--295, London
Math. Soc. Lecture Note Ser. 182,  Cambridge Univ. Press,
Cambridge, 1993.

\bibitem{Grom-random}
M. Gromov, {\em   Random walk in random groups}, Geom. Funct. Anal. {\bf 13} (2003), No. 1, 73--146.


\bibitem{Gu-Ma}
L. Guillou and  A. Marin (Eds.).
A la recherche de la topologie perdue.
Progress in Mathematics, Vol. 62,  Birkhauser, 1986.

\bibitem{HW}
F. Haglund and D.T.  Wise,
{\em A combination theorem for special cube complexes},
Ann. Math. (2) {\bf {176}} (2012), No. 3, 1427--1482.



\bibitem{HM}
S. Hermiller and J. Meier, {\em
Tame combings, almost convexity and rewriting systems for groups},  Math. Zeit. {\bf {225}} (1997), No. 2, 263--276.

\bibitem{Kap}
I. Kapovich, {\em
A note on the Po\'enaru condition}, J. Group Theory {\bf{5}} (2002), No. 1, 119--127.




\bibitem{McMill}
D.R. McMillan, {\em Some contractible open 3-manifolds}, Trans. Am. Math. Soc. {\bf{102}} (1962), 373--382.


 \bibitem{MT}
M.L. Mihalik and S.T. Tschantz, {\em Tame combings of groups},
 Tran. Am. Math. Soc. {\bf {349}} (1997), No. 10, 4251--4264.



\bibitem{MS}
C.F. Miller and M. Shapiro, {\em
Solvable Baumslag-Solitar groups are not almost convex},
Geom. Ded. {\bf {72}} (1998), No. 2, 123--127.


\bibitem{Mo-Ti}
J. Morgan, G. Tian. Ricci flow and the Poincar\'e conjecture.
Clay Mathematics Monographs 3. Providence, RI: American
Mathematical Society (AMS); Cambridge, MA: Clay Mathematics
Institute. xlii + 521 pp., 2007.


\bibitem{Mo-Ti2}
J. Morgan, G. Tian. The geometrization conjecture.
Clay Mathematics Monographs 5. Providence, RI: American Mathematical Society (AMS); Cambridge, MA: Clay Mathematics Institute. ix + 291 pp., 2014.


\bibitem{Mye}
R. Myers, {\em Contractible open 3-manifolds which are not covering spaces}, Topology  {\bf {27}} (1988), No. 1, 27--35.

\bibitem{N-Wh}
  M.H.A.  Newman, J.H.C.  Whitehead, {\em On the group of a certain linkage}, Quart. J. Math. Oxf. Ser. {\bf {8}} (1937),  14--21.

\bibitem{Ot_surv}
D.E. Otera, {\em Topological tameness conditions of spaces and groups:
Results and developments}, Lith. Math. J. {\bf{56}} (2016), No. 3,  357--376.




\bibitem{Ot-Po1}
D.E. Otera  and  V. Po\'enaru, {\em  ``Easy'' Representations and
the \textsc{qsf} property for groups}, Bull. Belgian Math. Soc. -
Simon Stevin   \textbf{19} (2012), No. 3, 385--398.


\bibitem{Ot-Po2} D.E. Otera  and  V. Po\'enaru,  {\em Tame combings and easy groups}, Forum Math.
{\bf {29}} (2017), No. 3, 665--680.



\bibitem{Ot-Po3} D.E. Otera  and  V. Po\'enaru,  {\em Finitely presented groups and the
Whitehead nightmare}, Groups Geom. Dyn. {\bf {11}} (2017), No. 1, 291--310.






\bibitem{Ot-Po-Ta}
D.E. Otera, V. Po\'enaru and C. Tanasi, {\em  On Geometric Simple
Connectivity}, Bulletin Math\'ematique de la Soci\'et\'e des
Sciences Math\'ematiques de Roumanie. Nouvelle S\'erie, Tome 53
(\textbf{101}) (2010), No. 2, 157--176.


\bibitem{Ot-Ru1}
D.E. Otera and F.G. Russo,
{\em On the \textsc{wgsc} property in some classes of groups},
Mediterr. J. Math. {\bf {6}} (2009), No. 4, 501--508.



\bibitem{Ot-Ru-Ta2}
D.E~ Otera, F.G. Russo and C. Tanasi,
{\em  On 3-dimensional \textsc{wgsc} inverse-representations of groups},
Acta Math. Hung. {\bf {151}} (2017), No. 2, 379--390.



\bibitem{Pap}
C.D. Papakyriakopoulos, {\em
On Dehn's lemma and the asphericity of knots},
Ann. Math. (2) {\bf {66}} (1957), 1--26.




\bibitem{Per1}
G. Perelman, {\em The entropy formula for the Ricci flow and its
geometric applications},  ArXiv: math/0211.159 [math.DG] (2002).



\bibitem{Per2}
G. Perelman, {\em Ricci flow with surgery on three-manifolds},
ArXiv: math/0303.109 [math.DG] (2003).



\bibitem{Per3}
G. Perelman, {\em Finite extinction time for the solutions of the
Ricci flow on certain three-manifolds}, ArXiv: math/0307.245
[math.DG] (2003).




\bibitem{Po1-duke}
V. Po\'enaru, {\em On the equivalence relation forced by the
singularities of a non-degenerate simplicial map}, Duke Math. J.
\textbf{63} (1991), No. 2, 421--429.




\bibitem{Po2-duke}
  V. Po\'enaru, {\em Killing
handles of index one stably and $\pi _1 ^\infty $}, Duke Math. J.
\textbf{63} (1991), No. 2,  431--447.




\bibitem{Po3-JDG}
V. Po\'enaru, {\em Almost convex groups, Lipschitz combing, and
$\pi_1 ^{\infty}$ for universal covering spaces of closed
$3$-manifolds}, J. Diff. Geom.  {\bf 35} (1992), No. 1,  103--130.



\bibitem{Po4-Top1}
V. Po\'enaru, {\em The collapsible pseudo-spine
 representation theorem}, Topology {\bf 31} (1992), No. 3, 625--656.




\bibitem{Po5-Top2}
V. Po\'enaru, {\em Geometry ``\`a la Gromov'' for the fundamental
group of a closed $3$-manifold $M^3$ and the simple connectivity
at infinity of $\widetilde M^3$},  Topology {\bf 33} (1994), No. 1,
181--196.



\bibitem{Po-contMath}
V. Po\'enaru, {\em $\pi_1^{\infty}$ and simple homotopy type in
dimension 3},  Contemporary Math. AMS,  {\bf 239} (1999), 1--28.


\bibitem{Po_pre0}
V.~Po\'enaru, {\em On the 3-Dimensional Poincar\'e Conjecture and
the 4-Dimensional Smooth Schoenflies Problem},
arXiv: math/0612554 [math.GT], 2006.


\bibitem{Po_QSF_survey}
V. Po\'enaru, {\em Discrete symmetry with compact fundamental
domain, and geometric simple connectivity - A provisional Outline
of work in Progress},  arXiv: 0711.3579 [math.GT].





\bibitem{Po_QSF1_Geom-Ded}
 V. Po\'enaru, {\em Equivariant, locally finite inverse
representations with uniformily bounded zipping length, for
arbitrary finitely presented groups}, Geom. Ded.
{\bf{167}} (2013), 91--121.



\bibitem{Po_QSF2}
V. Po\'enaru, {\em Geometric simple connectivity and finitely
presented groups}, arXiv: 1404.4283 [math.GT].



\bibitem{Po_QSF3}
V. Po\'enaru, {\em All finitely presented groups are
\textsc{qsf}},  arXiv: 1409.7325 [math.GT].





\bibitem{Po_GSC}
V. Po\'enaru, {\em Geometric simple connectivity and low-dimensional topology},
Geometric topology and set theory. Collected papers. Dedicated to the 100th birthday of Professor Lyudmila Vsevolodovna Keldysh. Transl. from the Russian. Moscow: Maik Nauka/Interperiodica. Proceedings of the Steklov Institute of Mathematics 247, 195--208 (2004) and Tr. Mat. Inst. Steklova 247, 214--227 (2004).



\bibitem{Po_4-dim}
V. Po\'enaru, {\em A glimpse into the problems of the fourth dimension}, Proceedings of the conference ``Geometry in History'', to appear.




\bibitem{PT2-k-weak}
V. Po\'enaru and C. Tanasi, {\em $k$-weakly almost convex groups
and $\pi_1^\infty \tilde M ^ 3$},  Geom. Ded. {\bf 48} (1993),
No. 1, 57--81.



\bibitem{PT-Wh}
V. Po\'enaru and C. Tanasi,
{\em Representatioin of the Whitehead manifold $Wh^3$ and Julia
sets}, Ann. Fac. Sci. Toulouse, S\'er. 6 {\bf 4} (1995), No. 3, 665--694.



\bibitem{PT-ActaHung}
V. Po\'enaru and C. Tanasi, {\em Some remarks on geometric simple
connectivity}, Acta Math. Hung. \textbf{81} (1998), No. 1-2, 1--12.



\bibitem{PT4-AMS}
V. Po\'enaru and  C. Tanasi, {\em Equivariant, Almost-Arborescent
representations of open simply-connected 3-manifolds; A finiteness
result}, \rm  Mem. AMS, {\bf 800}, 88 pp., 2004.


\bibitem{Po-Ta_Geom-Ded}
V. Po\'enaru and C. Tanasi, {\em A group-theoretical finiteness theorem}, Geom. Ded.  \textbf{137} (2008), No. 1, 1--25.


\bibitem{S-W}
A. Shapiro and J.H.C. Whitehead, {\em
A proof and extension of Dehn's lemma},
Bull. Am. Math. Soc. {\bf{64}} (1958), 174--178.



\bibitem{S-S}
M. Shapiro and M. Stein, {\em
Almost convex groups and the eight geometries},
Geom. Ded. {\bf {55}} (1995), No. 2, 125--140.

\bibitem{Sie}
L.C. Siebenmann,   {\em On detecting Euclidean space homotopically
among topological manifolds}, Invent. Math., {\bf 6} (1968),
263--268.


\bibitem{St1}
J.R. Stallings, {\em The piecewise linear structure of the
Euclidean space}, Proc.  Camb. Math. Philos. Soc.
{\bf{58}} (1962), 481--488.



\bibitem{St2}
J.R. Stallings, {\em  On torsion-free groups with infinitely many
ends}, Ann.  Math. (2) {\bf{88}} (1968), 312--334.


\bibitem{St_libro}
J.R. Stallings. Group theory and three-dimensional manifolds.
Yale Mathematical Monographs. 4. New Haven-London: Yale University Press, 1971.




\bibitem{St3}
J.R. Stallings, {\em Brick's quasi-simple filtrations for groups
and $3$-manifolds},   Geometric group theory, Vol. 1 (Sussex,
1991),  188--203, London Math. Soc. Lecture Note Ser.,
\textbf{181}, Cambridge Univ. Press, Cambridge, 1993.




\bibitem{Su-Thu}
D. Sullivan and W.P. Thurston, {\em
Manifolds with canonical coordinate charts: Some examples},
Enseign. Math. II. S\'er. {\bf {29}} (1983), 15--25.




\bibitem{Ta}
C. Tanasi, {\em Sui gruppi semplicemente connessi all'infinito},
Rend. Ist. Mat. Univ. Trieste {\bf{31}} (1999), No. 1-2, 61--78.




\bibitem{Thi}
C. Thiel,
{\em Zur fast-Konvexitat einiger nilpotenter Gruppen},
Universitat Bonn Mathematisches Institut, Bonn (1992)
Dissertation, Rheinische Friedrich-Wilhelms-Universitat Bonn, Bonn, 1991

\bibitem{Thu}
W.P. Thurston. Three-dimensional geometry and topology. Vol. 1.
Princeton Mathematical Series, 35. Princeton University Press,
Princeton, NJ, 1997.



\bibitem{Tu}
T.W. Tucker, {\em Non-compact $3$-manifolds and the missing
boundary problem}, Topology {\bf{73}} (1974), 267--273.

\bibitem{Wall_0}
C.T.C. Wall, {\em
Open 3-manifolds which are 1-connected at infinity},  Quart. J. Math. Oxf. Ser. (2) (1965), No. {\bf 16},  263--268.

\bibitem{Wall}
C.T.C. Wall, {\em Geometrical connectivity. I, II}, J. Lond. Math. Soc. (2) {\bf {3}} (1971),  597--604; 605--608.


\bibitem{Wh}
J.H.C. Whitehead, {\em  A certain open manifold whose group is
unity}, Quart. J. Math. Oxf. Ser.  {\bf{6}} (1935), 268--279.

\bibitem{Wri}
D.G. Wright, {\em Contractible open manifolds which are not covering spaces},
Topology {\bf{31}} (1992), No. 2, 281--291.



\end{thebibliography}
\end{document}